\documentclass{amsart}
\usepackage{amsmath,amsthm}
\usepackage{amsfonts,amssymb}
\usepackage{amsfonts}
\usepackage[all]{xy}
\usepackage{pdfsync}
\usepackage{hyperref}
\usepackage{graphicx}
\usepackage{subfigure}
\usepackage{color}
\usepackage{booktabs}
 \usepackage{multirow} 
\usepackage[title]{appendix}

\usepackage{enumitem}

\newtheorem{theorem}{Theorem}[section]
\newtheorem{lemma}[theorem]{Lemma}

\theoremstyle{definition}
\newtheorem{definition}[theorem]{Definition}

\theoremstyle{remark}
\newtheorem{remark}[theorem]{Remark}

\numberwithin{equation}{section}

\allowdisplaybreaks[4]

\begin{document}
 \title[Carleman estimates for higher step Grushin operators]{Carleman estimates for higher step Grushin operators}



\author[H. De Bie]{Hendrik De Bie}
\address{Clifford research group \\Department of Electronics and Information Systems \\Faculty of Engineering and Architecture\\Ghent University\\Krijgslaan 281, 9000 Ghent\\ Belgium.}
\curraddr{}
\email{Hendrik.DeBie@UGent.be}
\thanks{}

\author[P. Lian]{Pan Lian}
\address{1. School of Mathematical Sciences\\ Tianjin Normal University\\
Binshui West Road 393, Tianjin 300387\\ P.R. China\\
2. Clifford research group \\Department of Electronics and Information Systems \\Faculty of Engineering and Architecture\\Ghent University\\Krijgslaan 281, 9000 Ghent\\ Belgium.}
\curraddr{}
\email{panlian@tjnu.edu.cn}
\thanks{}

\keywords{Carleman estimate, Grushin operator, Jacobi polynomial, spherical harmonics,  unique continuation, Fischer decomposition}
\subjclass[2020]{Primary 35H20, 35B60; Secondary 33C45}

\date{}

\dedicatory{}


\begin{abstract} The higher step Grushin operators $\Delta_{\alpha}$ are a family of sub-elliptic operators which degenerate on a sub-manifold of $\mathbb{R}^{n+m}$. This paper establishes Carleman-type inequalities for these operators. 
It is achieved by deriving  a weighted $L^{p}-L^{q}$ estimate for the Grushin-harmonic projector. The crucial ingredient in the proof  is  the addition formula for Gegenbauer polynomials due to T. Koornwinder and Y. Xu. As a consequence, we obtain the strong unique continuation property for the Schr\"odinger operators $-\Delta_{\alpha}+V$ at points of  the degeneracy manifold, where $V$ belongs to certain $ L^{r}_{{\rm loc}}(\mathbb{R}^{n+m})$.
 \end{abstract}

\maketitle
\tableofcontents
\section{Introduction}

The study of unique continuation properties for partial differential operators traces back to Carleman's seminal work \cite{car} in 1939. Carleman proved that the time independent Schr\"odinger operator $\mathcal{H}=-\Delta+V$ in $\mathbb{R}^{2}$ with potential $V\in L^{\infty}_{{\rm loc}}(\mathbb{R}^{2})$, has the strong unique continuation property. Namely, $u\equiv 0$ is the unique solution to $\mathcal{H}u=0$ that has infinite order of vanishing at a point in $\mathbb{R}^{2}$. Since then, many unique continuation results  have been established for elliptic operators. This development  culminated with Jerison and Kenig's celebrated result \cite{JK}, which establishes the strong unique continuation property for $\mathcal{H}$ in $\mathbb{R}^{n}, n\ge 3$, where $V\in L_{{\rm loc}}^{n/2}(\mathbb{R}^{n})$. 
The consideration of  unbounded potentials $V$ is prompted by situations of physical interests. Independent of the Carleman's estimates,  an alternative geometric approach for unique continuation was developed  by Garofalo and Lin around 1986 in their profound works \cite{gl1, gl2}. Their method is based on local doubling properties proved by the Almgren monotonicity formula. There exists extensive literature on  unique continuation. For a comprehensive overview, we refer e.g. to the survey  \cite{tat}  and the references therein.


The unique continuation property for sub-elliptic operators is, in general, not true. A detailed understanding of the limitations or exceptions to this property is crucial for the broader comprehension of the behavior of sub-elliptic operators. Bahouri in \cite{Bah}  showed that unique continuation does not hold even for smooth and compactly supported perturbations of the  sub-Laplacian on the Heisenberg group $\mathbb{H}^{n}$. The closely related sub-elliptic operator (see \cite{Ga} for their connection), known as the Grushin operator $\Delta_{\alpha}$, is defined by 
\begin{equation} \label{ffg1}
    \Delta_{\alpha}:=\Delta_{x}+|x|^{2\alpha}\Delta_{y}, \qquad (x,y)\in \mathbb{R}^{n}\times \mathbb{R}^{m}, \,\alpha>0,
\end{equation}
where $\Delta_{x}$ is the ordinary Laplace operator on $\mathbb{R}^{n}$.
This operator was first introduced by Baoendi in \cite{baou} and later studied by Grushin in \cite{Gru1,Gru2},  where  hypoellipticity was obtained when $\alpha \in \mathbb{N}$. The subclass with $\alpha\in \mathbb{N} \backslash \{1\}$ is now commonly called the higher step Grushin operator. Geometrically, the higher step $\Delta_{\alpha}$ arises via a submersion from a sub-Laplace operator on a nilpotent  Lie group of step $\alpha+1$, see e.g. \cite{BFI}. Over the last years,
many classical problems in harmonic analysis, such as restriction estimates and spectral multipliers theorems,   have been 
successfully addressed   in the Grushin framework, see e.g. \cite{ccm1, ccm2, jst}. Moreover, the Grushin operator also has 
fruitful connections with other areas of study, such as the obstacle and extension problems for the fractional Laplacian, see e.g. the influential works \cite{caf,cars}. 

In \cite{Ga}, Garofalo initiated the study of strong unique continuation for zero-order perturbations of the  Grushin operator, i.e. $-\Delta_{\alpha}+V$, where $\alpha>0$ and the potential $V$ is subjected to the condition
\begin{equation} \label{pot1}
    |V(x, y)|\le C\psi(x,y).
\end{equation}
Here $\psi$ is the so-called angle function, which degenerates on $\{0\}\times \mathbb{R}^{m}$, see the subsequent formula \eqref{ps1}  for its definition. 
Note that the condition \eqref{pot1} forces the potential $V$ to vanish at $x=0$. Therefore, the permissible potentials are merely a subset of  $L^{\infty}_{\rm loc}$.  The results in \cite{Ga} were subsequently extended  to variable
coefficients Grushin operators in \cite{Gv}.  

Inspired by Jerison's work \cite{JD}, Garofalo and Shen in 1994  established  a delicate $L^{p}-L^{q}$ Carleman estimate and thus obtained the strong continuation property in \cite{GS} for the significant specific case where $n\ge 2$, $m=1$ and $\alpha=1$,  i.e.
\begin{equation*}
    \Delta_{x}u+|x|^{2}\partial_{t}^{2}u=Vu, \qquad (x, t) \in \mathbb{R}^{n}\times \mathbb{R},
\end{equation*}
 when $V\in L_{{\rm loc}}^{r}\left(\mathbb{R}^{n+1}\right)$, with 
 \begin{eqnarray}\label{jdq1}
\left\{
  \begin{array}{ll}
    r>n, & \hbox{if $n$ is even;} \\
   r>2n^{2}/(n+1), & \hbox{if $n$ is odd.}
  \end{array}
\right.
\end{eqnarray}
Notably, the condition \eqref{pot1} in terms of the degenerate weight $\psi$  is removed. The unresolved cases with $\alpha \neq 1$ and $m\neq 1$ are more challenging.  Considerable efforts have been devoted to this problem over the last years, see for instance the work of Garofalo and his collaborators \cite{BGM, BM}. Moreover, as was pointed out in \cite{BM},  the main difficulty is to establish a suitable estimate for the Grushin-harmonic projector, which will be addressed in  subsequent Section \ref{421}.  On the other hand, it is  known by many researchers that even for the specific  cases with $\alpha=1$ and $m>1$, i.e. the unique continuation property for equations 
 \begin{equation*}
     \Delta_{x}u+|x|^{2}\Delta_{y}u=Vu, \qquad (x, y)\in \mathbb{R}^{n}\times\mathbb{R}^{m},
 \end{equation*}
takes on significance. For instance, it will lead to interesting applications to unique continuation properties  for the sub-Laplacians on $H$-type groups, see e.g.\,\cite[Section 6]{GR} and \cite{GTr}. 
 
 
  In our work, following Garofalo and Shen \cite{GS}, we establish  Carleman estimates  for the zero-order perturbations of higher step Grushin operators, which lead to the strong unique continuation property for the Schr\"odinger operators $-\Delta_{\alpha}+V$, where $V$ is in  certain $L^{r}_{{\rm loc}}(\mathbb{R}^{n+m})$. Thus the restriction in terms of the degenerate weight $\psi$ in \eqref{pot1} is successfully removed. Interestingly,  although the degeneracy of $\Delta_{\alpha}$ becomes increasingly stronger as $\alpha\rightarrow \infty$,  the strong unique continuation property still holds for $-\Delta_{\alpha}+V$ at points of $\{0\}\times \mathbb{R}^{m}$ for all $V\in L^{\infty}_{{\rm loc}}$. It is worth noting that we were also able to remove the  distinction between even and odd dimensions $n$ in \eqref{jdq1} from \cite{GS}. Consequently, the results in \cite{GS} for odd dimensions are improved. We collect the exponents $p$ and $q$ in Carleman estimates and the potential spaces valid for strong  unique continuation  in Table \ref{tab:example2}.\\

  Now we discuss the main novelty in our derivation of  the Carleman estimates. 
  
  {\rm (\uppercase\expandafter{\romannumeral1}) {\bf Orthogonal basis for Grushin-harmonic polynomials}  
  
  The first key step is  the explicit construction of an orthogonal basis for Grushin-harmonic polynomials in terms of Jacobi polynomials. The structure of Grushin spherical harmonics  became  clear  when we recognized the underlying $\mathfrak{sl}_{2}$ structure through the connection between the Grushin operator and the so-called radially deformed Laplace operator, which arises from the study of the minimal representations of ${\rm O}(n,2)$ in \cite{bko}, see Section \ref{s21}. The significant  $\mathfrak{sl}_{2}$ structure   allows us to generalize classical results in complex and harmonic analysis, such as Fischer decomposition,   to the Grushin setting. Additionally, it is found that the Grushin-harmonics are also  related to the theory of the  Dunkl operator of rank 1 in \cite{dY}. When $\alpha=0$, we recover the  addition formula for Gegenbauer polynomials, which was  originally established by  to Koornwinder in \cite{koo} and later generalized by Xu in \cite{Xu}.

  {\rm (\uppercase\expandafter{\romannumeral2})} {\bf Estimates for Grushin-harmonic projection operator}
  
  The $L^{1}-L^{\infty}$ estimates for the Grushin-harmonic projection operator are
  derived using the aforementioned addition formula for Gegenbauer polynomials. The weighted $L^{2}-L^{2}$ estimates are  reduced to  uniform estimates for the weighted $L^{2}$ norms of Jacobi polynomials. The latter is achieved via two different  approaches: by the Bernstein-type inequality  recently obtained in \cite{HS} and by the connection formula for Jacobi polynomials  respectively. Our proof is concise, elementary, and distinctly different from the method of Garofalo and Shen in \cite[Lemma 4.3]{GS}.  This part is the main contribution of this paper. Moreover, we provide a tighter weighted $L^{2}-L^{2}$ estimate for the higher step Grushin-harmonic projector  on $\mathbb{R}^{n+1}$.
  
  We would  like to mention that the weighted $L^{2}$-norms of Jacobi and Gegenbauer polynomials  frequently appear in both physical and mathematical problems, for instance in the explicit computation of angular momentum, the generalization of Stolarsky's invariance principle to projective spaces, and  determinantal point processes,  see e.g.\,\cite{adgp, bg}. In the present framework,  the uniform norm estimates needed shall be viewed as a refined version of  Askey's famous transplantation theorem in the $L^{2}$ setting \cite{as}.

   Finally, combining the estimates obtained for projection operators in Section 4 with the Sobolev inequality for the Grushin operator of \cite{mor}, we seamlessly derive the Carleman estimates and the strong unique continuation property. The proof follows without significant difficulties from the approach  developed in Jerison's work \cite{JD}, which was later refined by Garofalo and Shen in the sub-elliptic setting \cite{GS}.

\begin{table}
        \centering
         \caption{Exponents in Carleman estimates and strong unique continuation}
         \label{tab:example2}
\begin{tabular}{cclccc}
   \toprule
      $m$ & $n$ &$\alpha$ & $p$  & $q$ &  $V\in L^{r}_{\rm loc}, r>$ \\
   \midrule
   \multirow{3}{*}{$m=1$} &   $n\ge 2$ &    1  & $\frac{2n}{n+1}$ &$\frac{2n}{n-1}$ & $n$\\  
    & $n\in\{2,3\}$  & $\mathbb{N}_{\ge 2}$  &  $\frac{6\alpha+10}{3\alpha+7}$ & $\frac{6\alpha+10}{3\alpha+3}$ & $(3\alpha+5)/2$\\
  & $n\ge 4$&  $\mathbb{N}_{\ge 2}$ &  $\frac{4+2(\alpha+1)(n-1)}{4+(\alpha+1)(n-1)}$  &$\frac{4+2(\alpha+1)(n-1)}{(\alpha+1)(n-1)}$&   $(2+(n-1)(\alpha+1))/2$ \\
   \hline
    \multirow{3}{*}{$m\ge 2$} 
    &$n=3$ & $\mathbb{N}_{\ge 2}$&$\frac{2\left(m+2+\frac{1}{\alpha+1}\right)}{m+2+\frac{2}{\alpha+1}}$ &$\frac{2\left(m+2+\frac{1}{\alpha+1}\right)}{m+2}$& $(\alpha+1)(m+2)+1$ \\
    \cline{2-6}
    &   $n\in\{2,3\}$ & $1$ & \multirow{2}{*}{$\frac{2\left(n+m-2+\frac{1}{\alpha+1}\right)}{n+m-2+\frac{2}{\alpha+1}}$}  & \multirow{2}{*}{$\frac{2\left(n+m-2+\frac{1}{\alpha+1}\right)}{n+m-2}$}& \multirow{2}{*}{$(\alpha+1)(n+m-2)+1$}
    \\
    &$n\ge 4$&   $\mathbb{N}$ &  & &\\
   \bottomrule
\end{tabular}
\end{table}

The rest of this paper is organized as follows: In Section \ref{sec2}, we introduce basic notions  of the Grushin operator and Jacobi polynomials which will be used in this paper. Section \ref{31} is focused on Grushin spherical harmonics, including the orthogonal basis and the addition formula. Section \ref{sec4} is devoted to the bounds  for the Grushin-harmonic projector and contains the main technical results of this paper. Section \ref{sec5} addresses Carleman estimates, while Section \ref{sec6} is for the strong unique continuation property. Conclusions are given at the end of this paper.

\section{Preliminaries}\label{sec2}
\subsection{Grushin operator} In this subsection, we go through basic notions for Grushin operators. More details can be found in e.g. \cite{dl, Ga, liu}.

The Grushin space, denoted by $\mathbb{R}^{n+m}=\{(x, y): x\in\mathbb{R}^{n},  y\in \mathbb{R}^{m}\}$,
is the Carnot-Carath\'eodory space equipped with a system of vector fields
\begin{eqnarray} \label{v1}
X_{i}=\frac{\partial}{\partial x_{i}},\, i=1,\ldots, n, \qquad Y_{j}=|x|^{\alpha}\frac{\partial}{\partial y_{j}}, \,j=1, 2,\ldots, m.
\end{eqnarray}
 Here $\alpha>0$ is a given real number and $|x|=\left(\sum_{i=1}^{n}x_{i}^{2}\right)^{1/2}$ represents the standard Euclidean norm of $x$.
%
\begin{definition}
The Grushin operator $\Delta_{\alpha}$ is a differential operator on $\mathbb{R}^{n+m}$ defined by
\begin{eqnarray}\label{l1} \Delta_{\alpha}:=\sum_{i=1}^{n}X_{i}^{2}+\sum_{j=1}^{m}Y_{j}^{2}=
\Delta_{x}+|x|^{2\alpha}\Delta_{y}=\nabla_{\alpha}\cdot\nabla_{\alpha},
\end{eqnarray}
where $\nabla_{\alpha}$ is the horizontal (or Grushin) gradient, that is
\begin{eqnarray*}
\nabla_{\alpha}:(X_{1}, \ldots, X_{n}, Y_{1}, \ldots, Y_{m})=(\nabla_{x}, |x|^{\alpha}\nabla_{y}).
\end{eqnarray*}
A function is called  Grushin-harmonic if it is a solution of the equation $\Delta_{\alpha}u=0$.
\end{definition}

\begin{remark}\begin{enumerate}
\item When $\alpha=0,$ the Grushin operator $\Delta_{\alpha}$ reduces to the standard Laplacian on the Euclidean space $\mathbb{R}^{n+m}$.  In the case where $\alpha=1$,  $m=1$ and $n$ even, it is closely related to the Kohn Laplacian on the Heisenberg group $\mathbb{H}^{\frac{n}{2}}$, see \cite{Ga, GS}.
                \item  For $\alpha>0$, the Grushin  operator  $\Delta_{\alpha}$ is elliptic when $x\neq 0$ and degenerate on the characteristic submanifold $\{0\}\times \mathbb{R}^m$ of $\mathbb{R}^{n+m}$.
                \item  Grushin studied $\Delta_{\alpha}$ in \cite{Gru1,Gru2} for the case where $\alpha$ is a non-negative integer and established  its hypoellipticity.
                \item  When $\alpha=2k$ with $k\in \mathbb{N}$, the Grushin operator
 $\Delta_{\alpha}$ is a sum of squares of $C^\infty$ vector fields that satisfy H\"ormander finite rank
condition on the Lie algebra $ Lie[X_{1}, \ldots, X_{n}, Y_{1}, \ldots, Y_{m}]$, see 
\cite{hor}.
              \end{enumerate}
\end{remark}

There exists a natural family of anisotropic dilations  attached to the Grushin operator $\Delta_{\alpha}$, given by
\begin{eqnarray}\label{f1}\delta_{\lambda}(x,y):=(\lambda x, \lambda^{\alpha+1}y), \qquad \lambda>0,\, (x,y)\in \mathbb{R}^{n+m}. \end{eqnarray}
From Eq.\,\eqref{f1}, we see that the degeneracy of  $\Delta_{\alpha}$ becomes increasingly stronger as $\alpha\rightarrow \infty$.

The change of variable formula for Lebesgue measure yields
\[{\rm d}\circ\delta_{\lambda}(x, y)=\lambda^{Q}\,{\rm d}x\, {\rm d}y,\]
where \begin{equation} \label{hdm}
    Q=n+m(\alpha+1).
\end{equation}
The number $Q$ is commonly referred to as the homogeneous dimension related to the vector fields \eqref{v1}. It  plays an important role in the local analysis of $\Delta_{\alpha}$ at points of  the degeneracy manifold $\{0\}\times \mathbb{R}^{m}$. 

Let us now introduce homogeneous functions in this framework, see e.g. \cite{Ga}.

\begin{definition} Suppose $u$ is a function on $\mathbb{R}^{n+m}$. We say that $u$ is homogeneous of  degree $k\in \mathbb{N}$ with respect to \eqref{f1} (or $u$ is homogeneous of $\delta_{\lambda}$-degree  $k$) if for each $\lambda>0$, it holds that
\begin{eqnarray} \label{f2} u\circ\delta_{\lambda}=\lambda^{k}u.\end{eqnarray}
Equivalently, it can be  characterized by
\[\mathbb{E}_{\alpha} u=ku,\]
where $\mathbb{E}_{\alpha}$ is the generalized Euler's operator introduced  in \cite[Eq.\,(2.2)]{Ga}, given by
\begin{eqnarray}\label{eu1} \mathbb{E}_{\alpha}:=\sum_{i=1}^{n}x_{i}\frac{\partial}{\partial x_{i}}+(\alpha+1)\sum_{j=1}^{m}y_{j}\frac{\partial}{\partial y_{j}}.\end{eqnarray}
The smooth vector field $\mathbb{E}_{\alpha}$ is indeed the generator of the group $\{\delta_{\lambda}\}_{\lambda>0}$.
\end{definition}
\begin{remark} It is not difficult to verify that $X_{i}$ and $Y_{j}$ are homogeneous of degree one with respect to the dilations $\{\delta_{\lambda}\}_{\lambda>0}$ in \eqref{f1}. That is,
\[X_{i}\circ \delta_{\lambda}=\lambda\delta_{\lambda}\circ X_{i}, \qquad Y_{j}\circ\delta_{\lambda}=\lambda\delta_{\lambda}\circ Y_{j}.\] Hence, $\Delta_{\alpha}$ is homogeneous of degree two with respect to the dilations $\{\delta_{\lambda}\}_{\lambda>0}$. 
\end{remark}

Associated to the vector fields \eqref{v1}, for any $(x, y)\in \mathbb{R}^{n+m}$, the gauge norm  is defined by
\begin{eqnarray}\label{gn}\rho(x,y):=\left(|x|^{2(\alpha+1)}+(\alpha+1)^{2}|y|^{2}\right)^{\frac{1}{2(\alpha+1)}}. \end{eqnarray}
Furthermore, the gauge-balls and gauge-spheres with respect to $\rho$ centered at the origin with radius $r$ are defined  by
\[ B_{r}:=\left\{(x,y)\in \mathbb{R}^{n+m}|\rho(x,y)<r\right\}
\]
and \[\partial B_{r}:=\{(x,y)\in \mathbb{R}^{n+m}|\rho(x,y)=r\}.\]
Note that these definitions depend on the value of $\alpha.$

When $\alpha$ is a non-negative integer,  a Grushin-harmonic function $u(x,y)$ can be expressed using its analyticity, as shown in \cite{M}. Specifically, for $(x,y)$  near the  origin, one can write $u$ as a series:
\begin{eqnarray*}
u(x,y)=\sum_{k=0}^{\infty}u_{k}(x,y),
\end{eqnarray*}
 where each $u_{k}$ is a homogeneous polynomial of $\delta_{\lambda}$-degree $k$. This series converges uniformly and absolutely in some neighbourhood of the origin.  It is not difficult to see that each $u_{k}$ is also Grushin-harmonic. This means that every  Grushin-harmonic function can be represented as an infinite sum of homogeneous Grushin-harmonic polynomials near the origin.

Based on the aforementioned analyticity  considerations,  we will always assume that $\alpha$ is a non-negative integer in the rest of this paper, i.e. $\alpha\in \mathbb{N}_{0}:=\{0,1,2,\cdots\}$. Furthermore, these cases have an interesting background from geometry. As  mentioned earlier,  the Grushin operator $\Delta_{\alpha}$  with $\alpha\in \mathbb{N}_{0}$,   arises through a submersion from a sub-Laplace operator on a nilpotent  Lie group of step $\alpha+1$, see \cite[Proposition 3.1]{BFI}.


\subsection{Polar coordinates}
In this subsection, we introduce suitable polar coordinates to derive a
decomposition of the  Grushin operator $\Delta_{\alpha}$ in \eqref{l1}. These coordinates for vector fields were first introduced by Greiner for the Heisenberg group $\mathbb{H}^1$  in \cite{gre} and later extended by Dunkl to $\mathbb{H}^n$ in \cite{du}. They were later adapted to the Grushin setting in \cite{GS} for $\alpha=m=1$ and further studied for general cases, see e.g. \cite{dl, liu} for $\alpha>0$ and $ m\in \mathbb{N}$.

Let
\begin{eqnarray}\label{pc}
\left\{
  \begin{array}{ll}
    x_{1}=\rho\sin\phi(\sin^{2}\phi)^{-\frac{\alpha}{2\alpha+2}} \sin\theta_{1} \sin\theta_{2}\cdots \sin\theta_{n-2}\sin \theta_{n-1}, \\
    x_{2}=\rho\sin\phi(\sin^{2}\phi)^{-\frac{\alpha}{2\alpha+2}} \sin\theta_{1} \sin\theta_{2}\cdots \sin\theta_{n-2}\cos \theta_{n-1},  \\
   \vdots  \\
 x_{n}=\rho \sin\phi(\sin^{2}\phi)^{-\frac{\alpha}{2\alpha+2}}  \cos\theta_{1};\\
    y_{1}=\frac{1}{\alpha+1}\rho^{\alpha+1}\cos\phi\sin \beta_{1} \sin\beta_{2} \cdots \sin\beta_{m-2}\sin \beta_{m-1}, \\
    y_{2}=\frac{1}{\alpha+1}\rho^{\alpha+1}\cos\phi\sin \beta_{1} \sin\beta_{2} \cdots \sin\beta_{m-2}
\cos \beta_{m-1}, \\
   \vdots \\
    y_{m}=\frac{1}{\alpha+1}\rho^{\alpha+1}\cos\phi\cos\beta_{1},
  \end{array}
\right.
\end{eqnarray}
where $\rho$ is the gauge norm defined in \eqref{gn}$, \phi\in (a_{\phi}, b_{\phi}), 0\le\theta_{i}<\pi$, $i=1,\cdots, n-2$, $0\le\beta_{j}<\pi$,
$j=1,\cdots, m-2$, $0\le\theta_{n-1}<2\pi$ and $0\le \beta_{m-1}<2\pi$. Here $a_{\phi}$ and $b_{\phi}$ depend on $n$ and $m$, that is
\begin{eqnarray}\label{jdq}
\left\{
  \begin{array}{ll}
    \phi\in(0, \frac{\pi}{2}), & \hbox{if $n, m\ge 2$;} \\
    \phi\in (0, \pi), & \hbox{if $n\ge 2$ and $m=1$;} \\
    \phi\in (-\frac{\pi}{2}, \frac{\pi}{2}), & \hbox{if $n=1$ and $m\ge 2$;} \\
    \phi\in (0, 2\pi), & \hbox{if $n=m=1$.}
  \end{array}
\right.
\end{eqnarray}

Let $r_{1}=|x|$ and $r_{2}=|y|$. When $n,m\ge 2$, it is seen  from \eqref{pc} that
\begin{eqnarray}\label{pc1}
\left\{
  \begin{array}{ll}
    r_{1}=|x|=\rho\sin^{\frac{1}{\alpha+1}}\phi,  \\
    r_{2}=|y|=\frac{1}{\alpha+1}\rho^{\alpha+1}\cos\phi.
  \end{array}
\right.
\end{eqnarray}
A straightforward computation yields
\begin{eqnarray}
\frac{\partial(r_{1}, r_{2})}{\partial(\rho, \phi)}=
\begin{pmatrix}\sin^{\frac{1}{\alpha+1}}\phi & \frac{1}{\alpha+1}\rho\sin^{\frac{-\alpha}{\alpha+1}}\phi\cos\phi \\ \rho^{\alpha}\cos\phi& -\frac{1}{\alpha+1}\rho^{\alpha+1}\sin\phi\end{pmatrix}.
\end{eqnarray}
This gives
\begin{eqnarray*}
{\rm d}r_{1}{\rm d}r_{2}=\frac{1}{\alpha+1}\sin^{-\frac{\alpha}{\alpha+1}}\phi\,\rho^{\alpha+1}\,{\rm d}\rho\, {\rm d}\phi,
\end{eqnarray*}
and
\begin{equation}\label{mea} 
    \begin{split}
        {\rm d}x\,{\rm d}y=&\, r_{1}^{n-1}r_{2}^{m-1}\,{\rm d}r_{1}\,{\rm d}\omega_{1}\,{\rm d}r_{2}\,{\rm d}\omega_{2}\\
=&\, \frac{1}{(\alpha+1)^{m}}\rho^{(\alpha+1)m+(n-1)}\sin^{\frac{n}{\alpha+1}-1}\phi\cos^{m-1}\phi \\ &\times {\rm d} \rho \,{\rm d} \phi \,{\rm d} \omega_{1}\,{\rm d}\omega_{2},
    \end{split}
\end{equation}
where ${\rm d} \omega_{1}$ and ${\rm d} \omega_{2}$ are the Lebesgue measures on $\mathbb{S}^{n-1}$ and $\mathbb{S}^{m-1}$, respectively.

From Eq.\,\eqref{pc},  the gauge norm $\rho$ in \eqref{gn} and $\phi$ are
\begin{eqnarray*}
\left\{
  \begin{array}{ll}
    \rho=\left[r_{1}^{2(\alpha+1)}+(\alpha+1)^{2}r_{2}^{2}\right]^{\frac{1}{2(\alpha+1)}}, \\
    \phi=\arctan\frac{r_{1}^{\alpha+1}}{(\alpha+1)r_{2}}.
  \end{array}
\right.
\end{eqnarray*}
Again direct computation yields that
\begin{eqnarray} \label{lco}\,
\frac{\partial(\rho, \phi)}{\partial(r_{1}, r_{2})}=
\begin{pmatrix}\sin^{\frac{2\alpha+1}{\alpha+1}}\phi & \rho^{-\alpha}\cos\phi \\ (\alpha+1)\rho^{-1}\sin^{\frac{\alpha}{\alpha+1}}\phi\cos\phi& -(\alpha+1)\rho^{-(\alpha+1)}\sin\phi\end{pmatrix}.
\end{eqnarray}

Alternatively,  it is well-known that in the usual spherical coordinates,
\[\Delta_{x}=\frac{\partial^{2}}{\partial r_{1}^{2}}+\frac{n-1}{r_{1}}
\frac{\partial}{\partial r_{1}}+\frac{1}{r_{1}^{2}}\Delta_{\mathbb{S}^{n-1}},  \]
and
\[\Delta_{y}=\frac{\partial^{2}}{\partial r_{2}^{2}}+\frac{m-1}{r_{2}}
\frac{\partial}{\partial r_{2}}+\frac{1}{r_{2}^{2}}\Delta_{\mathbb{S}^{m-1}},  \]
where $\Delta_{\mathbb{S}^{n-1}}$ denotes the Laplace-Beltrami operator on $\mathbb{S}^{n-1}$. Let  $\psi(\phi)$ be the angle function defined by \begin{eqnarray} \label{ps1} \psi(\phi):=\sin^{\frac{2\alpha}{\alpha+1}}\phi=\left(\frac{r_{1}}{\rho}\right)^{2\alpha}.\end{eqnarray}
After a direct computation based on \eqref{lco}, the Grushin operator $\Delta_{\alpha}$ in \eqref{l1} can be expressed in the polar coordinates \eqref{pc} as
\begin{eqnarray}\label{ps}
\Delta_{\alpha}=\psi(\phi)\cdot\left(\frac{\partial^{2}}{\partial\rho^{2}}+\frac{Q-1}{\rho}\frac{\partial}{\partial\rho}
+\frac{1}{\rho^{2}}\Delta_{\sigma} \right),\end{eqnarray}
where $\sigma\in (\phi, \omega_{1}, \omega_{2})$, $\omega_{1}\in \mathbb{S}^{n-1}$, $\omega_{2}\in \mathbb{S}^{m-1}$, and
\begin{equation}\label{r1}
    \begin{split}
        \Delta_{\sigma}=&\, (\alpha+1)^{2}\frac{\partial^{2}}{\partial \phi^{2}}+(n+\alpha-1)(\alpha+1) \cot\phi\frac{\partial}{\partial\phi}\\
&-(m-1)(\alpha+1)^{2}\tan\phi\frac{\partial}{\partial\phi}
+\sin^{-2}\phi\,\Delta_{\mathbb{S}^{n-1}}\\&+(\alpha+1)^{2}
\cos^{-2}\phi\,\Delta_{\mathbb{S}^{m-1}},
    \end{split}
\end{equation}
 see also \cite[Eq.\,(2.3)]{liu}.
\begin{remark} From Eq.\,\eqref{ps}, we see that if a function $u$ depends only on the pseudo-distance $\rho$, i.e. $u(x, y)=f(\rho)$, then 
\[\Delta_{\alpha}u=\sin^{\frac{2\alpha}{\alpha+1}}\phi \cdot \left(f''(\rho)+\frac{Q-1}{\rho}f'(\rho)\right). \]
 This implies that the Grushin operator does not map functions of $\rho$ into functions of $\rho$, which is different with the Euclidean Laplacian. This feature of the Grushin operator, which was shared by the sub-Laplacian on the Heisenberg group, makes the analysis considerably harder than that of the Euclidean setting. This has been pointed out in several studies, see e.g. \cite{BM, Ga, GS}. 
\end{remark}
\subsection{Jacobi polynomials} Gegenbauer polynomials have been used to construct the orthogonal basis for Grushin-harmonics on $\mathbb{R}^{n+1}$, see \cite{GS, liu}. In the general higher step setting, Jacobi polynomials will play a prominent role. We introduce several important properties of these polynomials in this subsection.

The Jacobi polynomials $P_{n}^{(\alpha, \beta)}(x)$ are a family of orthogonal polynomials defined on the interval $[-1,1]$ with respect to the the weight $(1-x)^{\alpha}(1+x)^{\beta}$. They can be expressed using the terminating  Gauss hypergeometric series as follows:
\begin{equation}\label{jad1}
  P_{n}^{(\alpha, \beta)}(x)=\frac{(\alpha+1)_{n}}{n!}\,  {_{2}F_{1}}\left(-n, 1+\alpha+\beta+n; \alpha+1; \frac{1-x}{2}\right),  
\end{equation}
where \[(\alpha+1)_{n}:=\frac{\Gamma(\alpha+1+n)}{\Gamma(\alpha+1)}\] is the Pochhammer symbol. When $\alpha>-1$ and $\beta>-1$, they  satisfy  the orthogonality condition, i.e.
\begin{equation}\label{or1}
    \begin{split}
        &\int_{-1}^{1}(1-x)^{\alpha}(1+x)^{\beta}P_{m}^{(\alpha,\beta)}(x)P_{n}^{(\alpha,\beta)}(x)\,{\rm d}x\\
=&\frac{2^{\alpha+\beta+1}}{2n+\alpha+\beta+1}\frac{\Gamma(n+\alpha+1)\Gamma(n+\beta+1)}{\Gamma(n+\alpha+\beta+1)n!}\delta_{n,m}.
    \end{split}
\end{equation}
Here $\delta_{n, m}$ is the Kronecker delta. Note that Jacobi polynomials include the Gegenbauer polynomials
\begin{eqnarray}\label{gen}
C_{n}^{(\alpha)}(x)=\frac{(2\alpha)_{n}}{\left(\alpha+\frac{1}{2}\right)_{n}}P_{n}^{(\alpha-1/2, \alpha-1/2)}(x)
\end{eqnarray}
as special cases. Furthermore,  the  Gegenbauer polynomials $C_{n}^{(\alpha)}(x)$ reduce to the Legendre polynomials when $\alpha=1/2$.

Extensive research has been conducted on the asymptotic behavior of Jacobi polynomials for large values of $n$, see e.g. \cite[Chapter VIII]{or}.  However,  most of the existing formulas do not provide uniform estimates for  $\alpha$ and $\beta$. An important  problem is to provide a uniform estimate for
\begin{eqnarray}
(1-x)^{a}(1+x)^{b}\left|P_{n}^{(\alpha, \beta)}(x)\right|
\end{eqnarray}
over the entire interval $[-1,1]$, where $a$ and $b$ are non-negative constants. The first result of this type is Bernstein’s inequality for the Legendre polynomials (see e.g. \cite[Theorem 7.3.3]{or}). It can be stated as follows:
\begin{eqnarray}\label{le}
(1-x^{2})^{1/4}\left|P_{n}(x)\right|\le \frac{2}{\sqrt{\pi(2n+1)}}, \qquad \forall\, x\in [-1, 1].
\end{eqnarray}
The constant $\sqrt{2/\pi}$ in \eqref{le} is known to be sharp. This problem has attracted considerable attention  over the last years, see e.g. \cite{emn, HS, kra}. For a comprehensive review of the progress made on this problem and more advanced estimates, we refer to \cite{KKT}.

 In this research, we need a uniform (in some sense) estimate for their weighted $L^{2}$-norms.
The following uniform estimate for Jacobi polynomials recently obtained in \cite{HS} will be used for that purpose. For $x\in [-1, 1]$ and $\alpha,  \beta\ge0$, let
\begin{equation}\label{g32}
    \begin{split}
       g_{n}^{(\alpha,\beta)}(x):=&\left(\frac{\Gamma(n+1)\Gamma(n+\alpha+\beta+1)}{\Gamma(n+\alpha+1)\Gamma(n+\beta+1)} \right)^{\frac{1}{2}}\\ &\times
\left(\frac{1-x}{2}\right)^{\frac{\alpha}{2}}\left(\frac{1+x}{2}\right)^{\frac{\beta}{2}}P_{n}^{(\alpha,\beta)}(x).
    \end{split}
\end{equation}
Then the following Bernstein-type inequality holds,
\begin{theorem}  There exists a constant  $C>0$ such that
\begin{eqnarray}\label{bs}
\left|(1-x^{2})^{\frac{1}{4}}g_{n}^{(\alpha,\beta)}(x)\right|\le C(2n+\alpha+\beta+1)^{-\frac{1}{4}}
\end{eqnarray}
for all $x\in [-1, 1]$, all $\alpha,\beta\ge0$, and all non-negative integers $n$.
\end{theorem}
\begin{remark} The optimal  value of $C$ is not yet known. However, it is known  at least that  $C<12$.
\end{remark}
 \begin{remark}
 With suitable coordinates, the functions $g_{n}^{(\alpha,\beta)}(x)$
with  non-negative integers $\alpha$ and $\beta$ form a natural and complete set of matrix coefficients for the irreducible
representations of ${\rm SU}(2)$. Interestingly, the value $2n + \alpha+ \beta + 1$ in \eqref{bs} is
exactly the dimension of the corresponding irreducible representation, see e.g. \cite[\S 9.14]{aar} and \cite[Theorem 2.1]{HS}. It follows that for all $x\in[-1, 1]$, $\alpha,\beta \in \mathbb{N}_{0}$ and $n\in \mathbb{N}_{0}$, the inequality
\begin{equation}\label{bou1}
    \left|g_{n}^{(\alpha,\beta)}(x)\right|\le 1
\end{equation}
holds.  This was studied in depth  and a tighter bound for $g_{n}^{(\alpha,\beta)}(x)$ was obtained, see  \cite[p. 234 (20)]{HS}.
\end{remark}}


\section{Spherical harmonics for the Grushin operator} \label{31}

In \cite{GS}, an orthogonal basis was constructed  for  Grushin-harmonic polynomials of the operator $\Delta_{\alpha}$ on  $\mathbb{R}^{n+1}$, where  $\alpha=1$,  using Gegenbauer polynomials \eqref{gen} and ordinary spherical harmonics. More recently, a similar orthogonal basis was discovered for general $\Delta_{\alpha}$ with $\alpha\in \mathbb{N}_{0}$ on $\mathbb{R}^{n+1}$  in \cite{liu}. Meanwhile, it  was observed in \cite{liu} that the Grushin harmonics on $\mathbb{R}^{n+m}$, when $m>1$,  exhibit increased complexity. In this section, we  address this issue for $\mathbb{R}^{n+m}$ with $n, m\ge2$. 

\subsection{Orthogonal  basis for  Grushin-harmonics}\label{shg}




Let $u_{k}$ be a homogeneous polynomial of $\delta_{\lambda}$-degree $k$ on $\mathbb{R}^{n+m}$. We write it in the polar coordinates \eqref{pc} as
 \begin{eqnarray}\label{gd1} u_{k}=\rho^{k}g(\phi, \omega_{1}, \omega_{2}),\end{eqnarray}
where $\omega_{1}\in \mathbb{S}^{n-1}$ and $\omega_{2} \in \mathbb{S}^{m-1}$.

We first give a specific family of Grushin-harmonic polynomials using Jacobi polynomials and ordinary spherical harmonics.

\begin{theorem}\label{th1} Suppose $n,m \ge 2$ and $\alpha, k\in \mathbb{N}_{0}$.  Let
 $Y_{\ell}(\omega_{1})$ and $Y_{j}(\omega_{2})$ be ordinary spherical harmonics of degree $\ell$ on $\mathbb{S}^{n-1}$ and degree $j$ on $\mathbb{S}^{m-1}$, respectively. Here $\ell$ and $j$ are non-negative integers satisfying the condition
\begin{eqnarray}\label{k1}\widetilde{k}:=\frac{k-\ell-j(\alpha+1)}{2(\alpha+1)}\in \mathbb{N}_{0}.\end{eqnarray}
Then,  the  function
\begin{eqnarray*}
u_{k}=\rho^{k}h_{k,\ell, j}(\phi)Y_{\ell}(\omega_{1})Y_{j}(\omega_{2}),
\end{eqnarray*}
where
\begin{eqnarray}\label{g1}h_{k, \ell, j}(\phi)=\cos^{j}\phi\,\sin^{\frac{\ell}{\alpha+1}}\phi\,
P_{\widetilde{k}}^{(\mu, \gamma-1)}(\cos 2\phi),
\end{eqnarray}
 with $\gamma=j+m/2$ and $\mu=\frac{n-2+2\ell}{2(\alpha+1)}$,  is  a homogeneous Grushin-harmonic polynomial of $\delta_{\lambda}$-degree $k$.
\end{theorem}
\begin{proof} Using formula \eqref{ps} for the Grushin operator $\Delta_{\alpha}$ in polar coordinates, a homogeneous polynomial $u_{k}$ of $\delta_{\lambda}$-degree $k$ is a solution of \[\Delta_{\alpha}u=0\]  if and only if $g(\phi, \omega_{1}, \omega_{2})$ in \eqref{gd1} is an eigenfunction of $\Delta_{\sigma}$ in \eqref{r1} with eigenvalue $-k(k+Q-2)$, i.e.
\begin{eqnarray}\label{f3}
\Delta_{\sigma}g=-k(k+Q-2)g,
\end{eqnarray}
where $Q=n+m(\alpha+1)$ is the homogeneous dimension.

According to \eqref{r1} and the well-known  property of ordinary spherical harmonics \begin{eqnarray*}
\left\{
  \begin{array}{ll}
    \Delta_{\mathbb{S}^{n-1}}Y_{\ell}(\omega_{1})=-\ell(\ell+n-2)Y_{\ell}(\omega_{1}), \\
    \Delta_{\mathbb{S}^{m-1}}Y_{j}(\omega_{2})=-j(j+m-2)Y_{j}(\omega_{2}),
  \end{array}
\right.
\end{eqnarray*}
it is seen that Eq.\,\eqref{f3} holds if and only if
\begin{equation}\label{s1}
    \begin{split}
        &
(\alpha+1)^{2}h''(\phi)+(n+\alpha-1)(\alpha+1) (\cot\phi) h'(\phi)\\ &-(m-1)(\alpha+1)^{2}(\tan\phi) h'(\phi)
+(\sin\phi)^{-2}(-\ell(\ell+n-2))h(\phi)\\ &+(\alpha+1)^{2}
(\cos\phi)^{-2}(-j(j+m-2))h(\phi)+k(k+Q-2)h(\phi)=0.
    \end{split}
\end{equation}

Letting $t=\cos\phi$ and $h(\phi)=v(t)$,  Eq.\,\eqref{s1} transforms into
\begin{equation} \label{jace}
    \begin{split}
        &(\alpha+1)^{2}(1-t^{2})\,\frac{{\rm d}^{2}}{{\rm d}t^{2}}v(t)-(\alpha+1)(n+2\alpha)t\,\frac{{\rm d}}{{\rm d}t}v(t)\\&+(m-1)(\alpha+1)^{2}\frac{1-t^{2}}{t}\,
\frac{{\rm d}}{{\rm d}t}v(t)+\frac{1}{1-t^{2}}(-\ell(\ell+n-2))v(t)
\\&+(\alpha+1)^{2}(-j(j+m-2))\frac{1}{t^{2}}v(t)+k(k+Q-2)v(t)=0.
    \end{split}
\end{equation}

Setting $w(t)=(1-t^{2})^{-\frac{\ell}{2(\alpha+1)}}v(t)$, then  $w(t)$ solves
\begin{equation}\label{e1}
    \begin{split}
        &t^{2}(1-t^{2})\frac{{\rm d}^{2}w(t)}{{\rm d}t^{2}}+t\left[\left(1-m-\frac{2\ell+n+2\alpha}{\alpha+1}\right)t^{2}+m-1\right]\frac{{\rm d}w(t)}{{\rm d}t}\\
-&\left[c(\alpha,\ell, k)t^{2}+j(j+m-2) \right] w(t)=0,
    \end{split}
\end{equation}
where \[c(\alpha,\ell, k)=\frac{\ell(\ell+n+(\alpha+1)m-2)-k(k+n+(\alpha+1)m-2)}{(\alpha+1)^{2}}.\]
One hypergeometric solution of  Eq.\,\eqref{e1}  is found to be 
\begin{eqnarray}\label{jp}&&w(t)\\
&=&t^{j}\, {_2F_{1}\left(\frac{j(a+1)-k+\ell}{2(\alpha+1)}, \frac{-2+j(\alpha+1)+k+\ell+m(\alpha+1)+n}{2(\alpha+1)}, j+\frac{m}{2}, t^{2}\right)}.\nonumber
\end{eqnarray}
It is a polynomial when
\[\widetilde{k}:=-\frac{j(a+1)-k+\ell}{2(\alpha+1)}\in \mathbb{N}_{0}, \]
or in other words when $ \ell+(\alpha+1)j\equiv k\left({\rm mod}(2\alpha+2)\right)$. Using \eqref{jad1},
the solution \eqref{jp} can be expressed   in terms of Jacobi polynomials, i.e.
\begin{eqnarray*}
w(t)&=&t^{j}\cdot{_2F_{1}\left(-\widetilde{k}, p+\widetilde{k}, \gamma, t^{2}\right)}\\
&=&t^{j}\frac{\widetilde{k}!}{(\gamma)_{\widetilde{k}}}P_{\widetilde{k}}^{(\gamma-1, p-\gamma)}(1-2t^{2})\\
&=&(-1)^{\widetilde{k}}t^{j}\frac{\widetilde{k}!}{(\gamma)_{\widetilde{k}}}P_{\widetilde{k}}^{(p-\gamma,\gamma-1)}(2t^{2}-1),
\end{eqnarray*}
where $\gamma=j+m/2$ and $p=\frac{n-2+2\ell}{2(\alpha+1)}+\gamma$. 
Here the second equality is by \eqref{jad1} and the last step is by the property $P_{j}^{(\alpha, \beta)}(-x)=(-1)^{j}P_{j}^{(\beta, \alpha)}(x)$.
This completes the proof.
\end{proof}
\begin{remark} The solutions \eqref{g1} can be verified directly by plugging them into  Eq.\,\eqref{jace} and then comparing with the equation with $\alpha=0$, which corresponds to the Euclidean case and is of course well-understood, see e.g. \cite[Eq.\,(4.3)]{koo}.
\end{remark}

In general, a second order ordinary differential equation possesses two linearly independent solutions. These solutions  can be combined to derive the general solution. The significance of the  hypergeometric solution \eqref{jp} selected in Theorem \ref{th1}, can be seen from the following lemma. This lemma follows from the basic properties of Jacobi polynomials and changing variables. The proof is thus omitted.
\begin{lemma}\label{hw1} Let $n, m\ge 2$. For fixed $\alpha, \ell, j\in\mathbb{N}_{0}$,
we denote $\gamma=j+m/2$ and $\mu=\frac{n-2+2\ell}{2(\alpha+1)}$ as in Theorem \ref{th1}. Then the set
\begin{eqnarray}\label{f4}\left\{h_{\widetilde{k}}(\phi):=\cos^{j}\phi\sin^{\frac{\ell}{\alpha+1}}\phi
P_{\widetilde{k}}^{(\mu, \gamma-1)}(\cos 2\phi)\right\}_{\widetilde{k}=0}^{\infty}\end{eqnarray}
forms a complete and  orthogonal basis of   $L^{2}\left([0,\pi/2], \sin^{\frac{n-2}{\alpha+1}+1}\phi\,\cos^{m-1}\phi \,{\rm d}\phi\right)$.
Moreover, it holds that
\begin{equation} \label{n1}
    \begin{split}
        &\int_{0}^{\pi/2}h_{\widetilde{k_{1}}}(\phi)h_{\widetilde{k_{2}}}(\phi)\sin^{\frac{n-2}{\alpha+1}+1}\phi\,\cos^{m-1}\phi \,{\rm d}\phi
\\=&\, \frac{1}{2(2\widetilde{k_{1}}+\mu+\gamma)}\frac{\Gamma(\widetilde{k_{1}}+\mu+1)\Gamma(\widetilde{k_{1}}+\gamma)}
{\Gamma(\widetilde{k_{1}}+\mu+\gamma)\Gamma(\widetilde{k_{1}}+1)}\delta_{\widetilde{k_{1}},\widetilde{k_{2}}}, \quad \gamma>0,\, \mu>-1.
    \end{split}
\end{equation}
\end{lemma}
\begin{remark} Note that the interval $[0, \pi/2]$  is important. Indeed, the set $\{h_{\widetilde{k}}(\phi)\}_{\widetilde{k}=0}^{\infty}$  does not form a complete basis for  $L^{2}\left([-\pi/2,\pi/2],  \sin^{\frac{n-2}{\alpha+1}+1}\phi\,\cos^{m-1}\phi \,{\rm d}\phi\right)$.  A complete basis in this case where $\phi\in [-\pi/2,\pi/2]$ would be obtained by employing the so-called generalized Gegenbauer polynomials in \cite[\S 1.5]{dY}. This suggests that we should make a distinction between the cases listed in \eqref{jdq}.
\end{remark}

To construct a complete orthogonal basis for Grushin-harmonic polynomials, we introduce some notations.
Let \[\{Y_{\ell, p}(\omega_{1})\}_{p=1,2,\ldots, d_{\ell}(n)}\]  be  an orthonormal  basis of ordinary spherical harmonics of degree $\ell$ on $\mathbb{S}^{n-1}$. Here $d_{\ell}(n)$ is   the dimension of the space of  ordinary spherical harmonics of degree $\ell$ on $\mathbb{S}^{n-1}$, which is given by
\begin{eqnarray}\label{dh1}
d_{\ell}(n)=\frac{(n+2\ell-2)\Gamma(n+\ell-2)}{\Gamma(\ell+1)\Gamma(n-1)}.
\end{eqnarray}
By Stirling's formula for the Gamma function,  there exists a constant $C(n)$ only depending on $n$ such that
\begin{eqnarray}\label{a1} d_{\ell}(n)\le C(n)(\ell+1)^{n-2}.\end{eqnarray} The set $\{Y_{j, q}(\omega_{2})\}_{q=1,2,\ldots, d_{j}(m)}$ is defined similarly. 

We introduce the following space spanned by homogeneous Grushin-harmonics obtained in Theorem \ref{th1}. 
\begin{definition}
We define
\begin{align*}
    \mathcal{H}_{k}^{\alpha}&:={\rm span} \bigg\{\rho^{k}h_{k,\ell, j}(\phi)Y_{\ell, p}(\omega_{1})Y_{j, q}(\omega_{2})\bigg|1\le p\le d_{\ell}(n), 1\le q\le d_{j}(m), \\ & \qquad 0\le \ell, j\le k, \, {\rm and} \,\, \ell+(\alpha+1)j\equiv k\left({\rm mod}(2\alpha+2)\right)\bigg\},
\end{align*}
where $h_{k, \ell, j}$ is given in \eqref{g1}. 
\end{definition}
Consider the measure on
\[\Omega:=\partial B_{1}=\{(x,y)\in \mathbb{R}^{n+m}|\rho(x,y)=1\}\]
given by
\begin{eqnarray} \label{dm1}
{\rm d}\Omega:=\sin^{\frac{n-2}{\alpha+1}+1}\phi\,\cos^{m-1}\phi \,{\rm d}\phi\, {\rm d}\omega_{1}\,{\rm d}\omega_{2},
\end{eqnarray}
which is from Lemma \ref{hw1}.
It is worth noting that ${\rm d}\Omega$ is not the usual Lebesgue measure on $\Omega$. It has the angular function $\psi(\phi)$ in \eqref{ps1} as a weight. This sheds light on the reason of the presence of $\psi$ in the condition \eqref{pot1}.\\

Now, we have all the ingredients to state the main result of this subsection:
\begin{theorem}\label{m1} Let $n, m\ge 2$.
The following orthogonal decomposition holds
\begin{eqnarray*}L^{2}(\Omega, {\rm d}\Omega)=\mathop{\bigoplus}_{k=0}^{\infty}\mathcal{H}^{\alpha}_{k}.
\end{eqnarray*}
\end{theorem}
\begin{proof}
Using the orthogonality  of the functions $h_{\widetilde{k}}(\phi)$ in \eqref{f4} and the orthogonality  of ordinary spherical harmonics of different degrees, we observe that the spaces $\mathcal{H}_{k}^{\alpha}$ are mutually orthogonal in $L^{2}(\Omega, {\rm d}\Omega)$. 
Thanks to  Lemma \ref{hw1} and the well-known  completeness of ordinary spherical harmonics in $L^{2}(\mathbb{S}^{n-1})$, the completeness of $\mathop{\bigoplus}_{k=0}^{\infty}\mathcal{H}_{k}^{\alpha}$  can be proven through a standard discussion. 
 In other words, we can show that if $f\in L^{2}(\Omega, {\rm d}\Omega)$ is orthogonal to each $\mathcal{H}_{k}^{\alpha}$, then $f=0$ a.e. on $\Omega$. We will not repeat the details.
\end{proof}

\begin{remark} \label{res1}
For the space $\mathbb{R}^{n+1}$, provided $\ell \equiv k({\rm mod}(\alpha+1))$, one polynomial solution of \eqref{e1}
is given by the Gegenbauer polynomial
\begin{eqnarray*}
w(t)=P^{\left(\frac{2\ell+n-2}{2(\alpha+1)},\frac{2\ell+n-2}{2(\alpha+1)}\right)}_{\frac{k-1}{\alpha+1}}(t)
\sim C^{\left(\frac{2\ell+n-2}{2(\alpha+1)}+\frac{1}{2}\right)}_\frac{k-\ell}{\alpha+1}
(t).\end{eqnarray*}
If we consider the following  measure on $\partial B_{1}$,
\begin{eqnarray*}
{\rm d}\Omega=\sin^{\frac{n-2}{\alpha+1}+1}\phi \,{\rm d}\phi\, {\rm d}\omega,
\end{eqnarray*}
where $\omega\in \mathbb{S}^{n-1}$, then an orthogonal basis of $\mathcal{H}^{\alpha}_{k}(\mathbb{R}^{n+1})$ is given by
\begin{eqnarray*}&&\left\{\rho^{k}\left(\sin\phi\right)^{\frac{\ell}{\alpha+1}} C^{\left(\frac{2\ell+n-2}{2(\alpha+1)}+\frac{1}{2}\right)}_\frac{k-\ell}{\alpha+1}
(\cos\phi)Y_{\ell, p}(\omega)\bigg|\right.\\ && \qquad  0\le p\le d_{\ell}(n), \ell \equiv k({\rm mod}(\alpha+1))\biggr\}_{k=0}^{\infty}, \end{eqnarray*}
where $\{Y_{\ell, p}\}_{p=1,2,\ldots d_{\ell}(n)}$ is an orthonormal basis of spherical harmonics of degree $\ell$ on $\mathbb{S}^{n-1}$. We recover the result in \cite[Theorem 1.2]{liu}. 
\end{remark}
 The proof of Theorem \ref{m1}  implies the following:

\begin{theorem} 
\label{decompharm} When $n, m\ge 3$,
the space $\mathcal{H}_{k}^{\alpha}$ decomposes under the action of ${\rm SO}(n)\times {\rm SO}(m)$ into irreducible and mutually orthogonal pieces as follows:
\begin{eqnarray*}
\mathcal{H}_{k}^{\alpha}=
\mathop{\bigoplus}_{\beta}h_{k,\ell, j}(\phi)\mathcal{H}_{\ell}(\mathbb{S}^{n-1})\otimes\mathcal{H}_{j}(\mathbb{S}^{m-1}),
\end{eqnarray*}
where $\mathcal{H}_{\ell}(\mathbb{S}^{n-1})$ is the space of ordinary spherical harmonics of degree $\ell$,  the condition
 $\beta$ means \[\ell+(\alpha+1)j\equiv k \,({\rm mod} (2\alpha+2)), \qquad 0\le \ell, j\le k,\]
and $h_{k, \ell, j}(\phi)$ is given in \eqref{g1}.

\end{theorem}
\begin{remark} When $\alpha=0$, Theorem \ref{decompharm} reduces to the decomposition of  ordinary  spherical harmonics by Koornwinder in \cite[Theorem 4.2]{koo}.
\end{remark}





\subsection{The orthogonal projection and addition formula}\label{re1}

In this subsection, we study the orthogonal projection
\begin{eqnarray}\label{GH}
P_{k}: L^{2}(\Omega, {\rm d}\Omega)\rightarrow\mathcal{H}^{\alpha}_{k},
\end{eqnarray}
which is the projection from $L^{2}(\Omega, {\rm d}\Omega)$ onto the $(k+1)$-th eigenspace of $\Delta_{\sigma}$ defined in \eqref{r1}.

Recall that for $n\ge 2$, the reproducing kernel of the ordinary homogeneous spherical harmonic  polynomials of degree $k$ over $\mathbb{S}^{n-1}$ is given by
\begin{eqnarray} \label{orp}
K^{n}_{k}(\zeta_{1},\zeta_{2})=\sum_{j=1}^{d_{k}(n)}Y_{k, j}(\zeta_{1})\overline{Y_{k,j}(\zeta_{2})}=\frac{d_{k}(n)}{\left|\mathbb{S}^{n-1}\right|}\cdot\frac{C_{k}^{(\frac{n-2}{2})}(\cos \tau)}{C_{k}^{(\frac{n-2}{2})}(1)}.
\end{eqnarray}
Here $\tau$ is the angle between $\zeta_{1}$ and $\zeta_{2}$  on $\mathbb{S}^{n-1}$, $0\le \tau\le \pi$, $C_{k}^{(\alpha)}(x)$ is the Gegenbauer polynomial  in \eqref{gen} and $d_{k}(n)$ is the dimension given in  \eqref{dh1}. See e.g.\,\cite[Theorem 9.6.3]{aar} and \cite{dY}.

For convenience, we denote
\begin{equation}\label{bab}
    \begin{split}
      B_{n}^{\alpha,\beta}:=&\int_{0}^{\pi/2}(\sin \phi)^{2\alpha+1}(\cos \phi)^{2\beta+1}\left(P_{n}^{(\alpha,\beta)}(\cos 2\phi)\right)^{2}{\rm d}\phi
\\=&
\frac{1}{2(2n+\alpha+\beta+1)}\frac{\Gamma(n+\alpha+1)\Gamma(n+\beta+1)}{\Gamma(n+\alpha+\beta+1)n!},
    \end{split}
\end{equation}
which is the square of the $L^{2}$-norm of the Jacobi polynomial $P^{(\alpha,\beta)}_{n}(x)$ in trigonometric coordinates.

Now, let $\omega=(\omega_{1}, \omega_{2}), \eta=(\eta_{1},\eta_{2}) \in \mathbb{S}^{n-1}\times\mathbb{S}^{m-1}$ and
${\rm d}\omega:={\rm d}\omega_{1} {\rm d}\omega_{2}$. It is not hard to see from Theorem \ref{m1} that the following holds.


\begin{theorem}\label{c1}   The orthogonal projector operator
\[P_{k}: L^{2}(\Omega, {\rm d}\Omega)\rightarrow\mathcal{H}^{\alpha}_{k}\]
coincides with the integral operator
\begin{eqnarray} \label{int1}
P_{k}(g)(\phi, \omega)&=&\int_{0}^{\pi/2}\int_{\mathbb{S}^{n-1}}\int_{\mathbb{S}^{m-1}}G_{k}(\phi,\omega, \xi, \eta)g(\xi, \eta)\\&&\times\sin^{\frac{n-2}{\alpha+1}+1}\xi\,\cos^{m-1}\xi\, {\rm d}\xi\, {\rm d}\eta,\nonumber
\end{eqnarray}
where
\begin{equation}\label{af1}
    \begin{split}
       &G_{k}(\phi,\omega, \xi, \eta)\\=&\mathop{\sum_{\widetilde{k}\in\mathbb{N}_{0}}}_{0\le\ell,j\le k}b_{k,\ell,j}^{2}h_{k,\ell,j}(\phi)h_{k,\ell,j}(\xi)
K_{\ell}^{n}(\omega_{1},\eta_{1})K_{j}^{m}(\omega_{2},\eta_{2}),
    \end{split}
\end{equation}
in which $b_{k,\ell,j}$ is the normalization constant given by
\begin{eqnarray}\label{nc1}
b_{k,\ell,j}=\left(\frac{1}{B_{\widetilde{k}}^{\mu,\gamma-1}}\right)^{\frac{1}{2}}=
\left(\frac{2(2\widetilde{k}+\mu+\gamma)\Gamma(\widetilde{k}+\mu+\gamma)\Gamma(\widetilde{k}+1)}{\Gamma(\widetilde{k}+\mu+1)\Gamma(\widetilde{k}+\gamma)}\right)^{\frac{1}{2}}.
\end{eqnarray}
Here $h_{k, \ell, j}(\phi), \widetilde{k}, \mu, \gamma$ are defined in Theorem \ref{th1} and $K_{k}^{n}(\cdot, \cdot)$ is given in \eqref{orp}.
\end{theorem}

When $\alpha=0$, the integral kernel $G_{k}(\phi,\omega, \xi, \eta)$ in \eqref{af1} reduces to the reproducing kernel $K^{n+m}_{k}(\cdot, \cdot)$ for the ordinary spherical harmonics of degree $k$ in \eqref{orp}. Thus in this case, we recover  the addition formula for Gegenbauer polynomials obtained first by  Koornwinder in \cite[Eq.\,(4.7)]{koo} using a group theoretical method  without obtaining explicit constants. In 1997, Xu provided a generalization of the formula and an analytic proof  in \cite[Eq.\,(2.5)]{Xu}. We adapt the notations to suit our purpose as follows.
\begin{theorem}\cite{Xu} \label{adfi}
For $u>1$ and $v> 1$,  the following addition formula holds
\begin{eqnarray}\label{af}
&&C_{k}^{\left(\frac{u+v}{2}-1\right)}(\cos \phi\cos \xi\cos\theta_{1}+\sin\phi\sin \xi\cos\theta_{2} )
\\&=&\sum_{i=0}^{\lfloor k/2\rfloor}\sum_{j+\ell=k-2i} c_{k,\ell,j}f_{k,\ell,j}(\cos 2\phi)f_{k,\ell,j}(\cos 2\xi)C_{j}^{(\frac{u}{2}-1)}(\cos\theta_{1})C_{\ell}^{(\frac{v}{2}-1)}(\cos\theta_{2}),\nonumber
\end{eqnarray}
where 
\begin{eqnarray*}f_{k,\ell,j}(\cos 2\theta)=\left(\frac{1}{B_{\frac{k-\ell-j}{2}}^{\frac{v}{2}-1+j,\frac{u}{2}-1+\ell}}\right)^{\frac{1}{2}}(\cos\theta)^{j}(\sin \theta)^{\ell}P_{\frac{k-\ell-j}{2}}^{(\frac{v}{2}-1+\ell, \frac{u}{2}-1+j)}(\cos 2\theta),\end{eqnarray*}
and
\begin{eqnarray*}c_{m,k,\ell}=\frac{(2j+u-2)(2\ell+v-2)\Gamma\left(\frac{u}{2}-1\right)\Gamma\left(\frac{v}{2}-1\right)}{4(2k+u+v-2)\Gamma\left(\frac{u+v}{2}-1\right)}.
\end{eqnarray*}
\end{theorem}
\begin{remark}  The addition formula  \eqref{af} will be crucial to establish the $L^{1}-L^{\infty}$ estimates for the projector $P_{k}$ in \eqref{GH} in the next section.
\end{remark}
\subsection{The $\mathfrak{sl}_{2}$ triple and Fischer decomposition}\label{s21}
In this subsection, we further  investigate Grushin-harmonics and the orthogonal projection  \eqref{GH}. The interesting  $\mathfrak{sl}_{2}$  structure  allows us to generalize the classical harmonic analysis results which are governed by $\mathfrak{sl}_{2}$  to the present framework. We only point out some of them.

Let $\mathcal{P}$ be the real vector space of all polynomials on $\mathbb{R}^{n+m}$,  and let  $\mathcal{P}_{k}^{\alpha}$ be 
the subspace consisting of homogeneous polynomials of $\delta_{\lambda}$-degree $k$. The subsequent result has been implicitly used in  literature. Since we have been unable to find a proof,  we  provide one here.

\begin{theorem}\label{cs1} For any $k\in \mathbb{N}_{0}$, the mapping $f\rightarrow \Delta_{\alpha}f$ establishes a  surjective homomorphism from $\mathcal{P}_{k}^{\alpha}$ to $\mathcal{P}_{k-2}^{\alpha}$. {\rm (}The spaces $\mathcal{P}_{-1}^{\alpha}$ and $\mathcal{P}_{-2}^{\alpha}$ are assumed to be $\{0\}$.{\rm )}
\end{theorem}
\begin{proof} It is not hard to see that $\Delta_{\alpha}$ preserves  polynomials and is  homogeneous of $\delta_{\lambda}$-degree two. Thus it is sufficient to show that this mapping is surjective. 

This is  true when $k\le \alpha$, because now the polynomials in $\mathcal{P}_{k}^{\alpha}$ are only those linear combinations of monomials $x_{1}^{k_{1}}\cdots x_{n}^{k_{n}}$ with $k_{1}+\cdots+k_{n}=k$ and hence we can use the result for the classical Laplace operator. 

Now, let us consider the general monomial $XY$, where
\begin{equation*}
    X:=x_{1}^{\ell_{1}}\cdots x_{n}^{\ell_{n}}, \qquad Y:=y_{1}^{j_{1}}\cdots y_{m}^{j_{m}}
\end{equation*}
with $\ell=\ell_{1}+\cdots+\ell_{n}$ and $j=j_{1}+\cdots+j_{m}$. It is seen that $XY$ is of $\delta_{\lambda}$-degree $\ell+(\alpha+1)j$. 
Let $X_{1}$ be a polynomial (not unique) such that $\Delta_{x} X_{1}=X$,  and $X_{\ell}$ be the polynomial such that $\Delta_{x} X_{\ell}=|x|^{2\alpha}X_{\ell-1}$ for $\ell\ge 2$. The existence of such $X_{\ell}$ follow from the Euclidean results. 
Then, 
\begin{equation*}
\begin{split}
    &(\Delta_{x}+|x|^{2\alpha}\Delta_{y})\sum_{\ell=1}^{\lfloor \frac{j+4}{2}\rfloor}(-1)^{\ell-1} X_{\ell} \Delta_{y}^{\ell-1}Y \\
    =&\, XY+\sum_{\ell=2}^{\lfloor \frac{j+4}{2}\rfloor}(-1)^{\ell-1}|x|^{2\alpha}X_{\ell-1} \Delta_{y}^{\ell-1}Y
    +\sum_{\ell=1}^{\lfloor \frac{j+4}{2}\rfloor}(-1)^{\ell-1} |x|^{2\alpha}X_{\ell} \Delta_{y}^{\ell}Y\\
    =&\, XY,
\end{split}
\end{equation*}
where we have used $\Delta_{y}^{\lfloor \frac{j+4}{2}\rfloor}Y=0$ in the last step. Since every polynomial is a linear combination of monomials, we complete  the proof.
\end{proof}

Now we give some dimension discussions. 
\begin{lemma}\label{dqq}
    The dimension of $\mathcal{P}_{k}^{\alpha}(\mathbb{R}^{n+m})$, denoted by ${\dim}\,\mathcal{P}_{k}^{\alpha}(\mathbb{R}^{n+m}) $, satisfies
\begin{equation}\label{dq1}
    \frac{1}{(1-r^{\alpha+1})^{m}(1-r)^{n}}=\sum_{k=0}^{\infty}{\dim}\,\mathcal{P}_{k}^{\alpha}(\mathbb{R}^{n+m}) r^{k}, \qquad |r|<1.
\end{equation}
Explicitly, suppose $k=p(\alpha+1)+\ell$ for some positive integer $p$  and $\ell\in \{0,1,\cdots, \alpha\},$ then we have 
 \begin{equation}\label{dq2}
     {\dim}\,\mathcal{P}_{k}^{\alpha}(\mathbb{R}^{n+m})=\sum_{j=0}^{p}\binom{m+j-1}{m-1}\cdot\binom{n+k-j(\alpha+1)-1}{n-1}.
 \end{equation} 
\end{lemma} 
\begin{proof} The formula \eqref{dq2}  was already appeared in \cite[Lemma 2.1]{liu}. It can be seen as follows. Since each $x_{i}, i=1,2, \ldots, n$ is homogeneous of $\delta_{\lambda}$-degree one and 
    $y_{j}, j=1,2, \ldots, m$ is homogeneous of $\delta_{\lambda}$-degree $\alpha+1$, thus a basis for 
    $\mathcal{P}_{k}^{\alpha}(\mathbb{R}^{n+m})$ is given by \begin{eqnarray*}\{p_{k_{1}}(x)q_{k_{2}}(y): p_{k_{1}}(x) \in \mathcal{P}_{k_{1}}(\mathbb{R}^{n}), q_{k_{2}}(y)\in \mathcal{P}_{k_{2}}(\mathbb{R}^{m}), k_{1}+k_{2}(\alpha+1)=k\}.
  \end{eqnarray*}
Using the well-known result for ordinary polynomials 
$
   {\rm dim}\mathcal{P}_{k}(\mathbb{R}^{n})=\binom{n+k-1}{n-1}, 
$ we are done.  
The formula \eqref{dq1} can be seen from \eqref{dq2} and the following well-known generating function 
\begin{equation*}
    \frac{1}{(1-r)^{n}}=\sum_{k=0}^{\infty}\binom{n+k-1}{n-1}r^{k}, \qquad |r|<1,
\end{equation*}
which generates the diagonals in Pascal's Triangle.
\end{proof}
By Theorem \ref{cs1} and Lemma \ref{dqq}, we have,
\begin{lemma} The dimension of $\mathcal{H}_{k}^{\alpha}(\mathbb{R}^{n+m})$, denoted by ${\dim}\,\mathcal{H}_{k}^{\alpha}(\mathbb{R}^{n+m}) $, satisfies 
 \begin{equation} \label{dq3}
        \frac{1-r^{2}}{(1-r^{\alpha+1})^{m}(1-r)^{n}}=\sum_{k=0}^{\infty}{\dim}\,\mathcal{H}_{k}^{\alpha}(\mathbb{R}^{n+m}) r^{k}, \qquad |r|<1.
 \end{equation}
  Explicitly, for $\alpha\neq 0$, let $k=p(\alpha+1)+\ell$ for some positive integer $p$  and $\ell\in \{0,1,\cdots, \alpha\},$  then we have
 \begin{equation*} \begin{split}
     {\dim}\,\mathcal{H}_{k}^{\alpha}(\mathbb{R}^{n+m})=
         &\sum_{j=0}^{p}\binom{m+j-1}{m-1}\cdot\binom{n+k-j(\alpha+1)-1}{n-1}\\
     &-\sum_{j=0}^{p}\binom{m+j-1}{m-1}\cdot\binom{n+k-j(\alpha+1)-3}{n-1},
     \end{split}
  \end{equation*}    
    when $\ell\ge 2$,
    and 
    \begin{equation*}
       \begin{split} {\dim}\,\mathcal{H}_{k}^{\alpha}(\mathbb{R}^{n+m})=&\sum_{j=0}^{p}\binom{m+j-1}{m-1}\cdot\binom{n+k-j(\alpha+1)-1}{n-1}\\
     &-\sum_{j=0}^{p-1}\binom{m+j-1}{m-1}\cdot\binom{n+k-j(\alpha+1)-3}{n-1},
     \end{split}  
    \end{equation*}
  when $\ell<2$.
\end{lemma}
\begin{proof} Using Theorem \ref{cs1}, we have 
    \begin{equation*}
        {\dim}\,\mathcal{H}_{k}^{\alpha}(\mathbb{R}^{n+m})={\dim}\,\mathcal{P}_{k}^{\alpha}(\mathbb{R}^{n+m})-{\dim}\,\mathcal{P}_{k-2}^{\alpha}(\mathbb{R}^{n+m}).
    \end{equation*}
   Together with  \eqref{dq1} and \eqref{dq2}, we obtain the explicit formulas.
\end{proof}
\begin{remark} Rewriting the left-hand side of \eqref{dq3} as 
\begin{equation*}
     \frac{1-r^{2}}{(1-r^{\alpha+1})^{m}(1-r)^{n}}= \frac{1}{1-r^{2(\alpha+1)}}\cdot \frac{1-r^{2(\alpha+1)}}{(1-r^{\alpha+1})^{m}} \cdot \frac{1-r^{2}}{(1-r)^{n}},
\end{equation*}
and then by the well-known relation for ordinary spherical harmonics (see \cite{mu}), i.e.
 \begin{equation*} 
        \frac{1-r^{2}}{(1-r)^{n}}=\sum_{k=0}^{\infty}{\dim}\,\mathcal{H}_{k}(\mathbb{R}^{n}) r^{k}, \qquad |r|<1,
 \end{equation*}
    we have 
    \begin{equation*}
        {\dim}\,\mathcal{H}_{k}^{\alpha}(\mathbb{R}^{n+m})=\sum_{\beta} {\dim}\,\mathcal{H}_{\ell}(\mathbb{R}^{n}) \cdot {\dim}\,\mathcal{H}_{j}(\mathbb{R}^{m}), 
    \end{equation*}
    where $\beta$ stands for $\ell+(\alpha+1)j\equiv k({\rm mod}(2\alpha+2)$, $0\le \ell, j\le k$ and ${\dim}\,\mathcal{H}_{\ell}(\mathbb{R}^{n})$ is the dimension of ordinary harmonic polynomials of degree $\ell$ on $\mathbb{R}^{n}$ given in \eqref{dh1}. This again implies that the orthogonal basis constructed in Section \ref{shg} is complete.
\end{remark}

 Theorems \ref{m1} and \ref{cs1} together  imply that the  classical Fischer decomposition, a fundamental tool in classical harmonic analysis and complex analysis (see e.g. \cite{dY}), can be extended to the present Grushin setting with $\alpha\in \mathbb{N}_{0}$.
\begin{theorem}[Fischer decomposition] \label{fis} The following orthogonal decomposition holds
\begin{eqnarray*}
\mathcal{P}_{k}^{\alpha}(\mathbb{R}^{n+m})=\mathop{\bigoplus}_{p=0}^{\lfloor \frac{k}{2} \rfloor}\rho^{2p}\,\mathcal{H}^{\alpha}_{k-2p}(\mathbb{R}^{n+m}).
\end{eqnarray*}  
\end{theorem}

Obviously, when $\alpha=0$, this reduces to the classical result. On the other side, the classical Fischer decomposition is well understood by employing the Howe dual pair $({\rm SO}(n+m), \mathfrak{sl}_{2})$, see e.g.\,\cite{Howe}. We show that the decomposition in Theorem \ref{cs1} is  again governed  by a  $\mathfrak{sl}_{2}$ triple. This will be achieved  by introducing  the following  perturbed  operator which maps functions of $\rho$ into functions of $\rho$. Meanwhile, the Grushin-harmonic  polynomials are  included in the kernel of this new operator.
\begin{definition} The perturbed Grushin operator $\mathcal{L}_{\alpha}$ is given by
\begin{eqnarray*}\mathcal{L}_{\alpha}&:=&\left(\frac{\rho}{|x|}\right)^{2\alpha}\Delta_{\alpha}=\rho^{2\alpha}\left(|x|^{-2\alpha}\Delta_{x}+\Delta_{y}\right)\\
&=&\frac{\partial^{2}}{\partial\rho^{2}}+\frac{Q-1}{\rho}\frac{\partial}{\partial\rho}
+\frac{1}{\rho^{2}}\Delta_{\sigma}.
\end{eqnarray*}
\end{definition}

A straightforward calculation yields the following,
\begin{theorem} \label{sl2} The operators $\mathcal{L}_{\alpha}$, $\rho^{2}$ and $\mathbb{E}_{\alpha}+Q/2$ form an $\mathfrak{sl}_{2}$ triple for any $\alpha\in\mathbb{N}_{0}$. Namely, the following commutations hold
 \begin{eqnarray*}
[\mathcal{L}_{\alpha}, \rho^{2}]&=&4\left(\mathbb{E}_{\alpha}+Q/2\right),\\
\left[\mathcal{L}_{\alpha}, \mathbb{E}_{\alpha}+Q/2\right]&=&2\mathcal{L}_{\alpha},\\
\left[\rho^{2}, \mathbb{E}_{\alpha}+Q/2\right]&=&-2\rho^{2},
\end{eqnarray*}
where $\mathbb{E}_{\alpha}$ is the analogue of classical Euler operator defined in \eqref{eu1}.
\begin{proof} We only  verify  the first relation in Cartesian coordinates. Direct calculation yields 
\begin{equation*}
    \begin{split}
       & \partial_{x_{i}}\rho^{2}=2\rho^{-2\alpha}|x|^{2\alpha}x_{i},\qquad  \qquad 1\le i\le n, \\
      & \partial_{y_{j}}\rho^{2}=2(\alpha+1)\rho^{-2\alpha}y_{j},\quad \qquad 1\le j\le m,\\
    &\partial_{x_{i}}^{2}\rho^{2}=-4\alpha\rho^{4\alpha-2}|x|^{4\alpha}x_{i}^{2}+4\alpha\rho^{-2\alpha}|x|^{2\alpha-2}x_{i}^{2}
    +2\rho^{-2\alpha}|x|^{2\alpha},\\
    &\partial_{y_{j}}^{2}\rho^{2}=-4\alpha(\alpha+1)^{2}\rho^{-4\alpha-2}y_{j}^{2}+2(\alpha+1)\rho^{-2\alpha}.
    \end{split}
\end{equation*}
It follows that
\begin{eqnarray*} 
        [\mathcal{L}_{\alpha}, \rho^{2}]&=& \rho^{2\alpha}\left[|x|^{-2\alpha}\Delta_{x}+\Delta_{y}, \rho^{2}\right]\\
        &=& \rho^{2\alpha}\left[|x|^{-2\alpha}\Delta_{x}, \rho^{2}\right]+\rho^{2\alpha}\left[\Delta_{y}, \rho^{2}\right]\\
        &=& 4\sum_{i=1}^{n}x_{i}\partial_{x_{i}}+2n+4\alpha-4\alpha\frac{|x|^{2(\alpha+1)}}{\rho^{2(\alpha+1)}}\\
        &&+4(\alpha+1)\sum_{j=1}^{m}y_{j}\partial_{y_{j}}+2m(\alpha+1)-\frac{4\alpha(\alpha+1)^{2}|y|^{2}}{\rho^{2(\alpha+1)}}\\
        &=& 4(\mathbb{E}_{\alpha}+Q/2).
\end{eqnarray*} 
These relations can also be seen by   using the polar coordinates.  
\end{proof}

\end{theorem}
\begin{remark} (1) The operator $\mathcal{L}_{\alpha}$ can be further deformed as $\rho^{2-b}\mathcal{L}_{\alpha}$ with $b>0$. Together with $\rho^{b}$ and $\mathbb{E}_{\alpha}+c$, they will give a  new family of realization of $\mathfrak{sl}_{2}$. In particular when $b=2+2\alpha$, $\rho^{2-b}\mathcal{L}_{\alpha}$ reduces to the radially deformed Laplace operator on $\mathbb{R}^{n+m}$ in \cite{bko} plus the ordinary Laplace operator, i.e. $|x|^{-2\alpha}\Delta_{x}+\Delta_{y}$, which is also called the split Laplace operator.

(2) The $\mathfrak{sl}_{2}$ structure makes it possible to generalize many results to the present setting. For instance, a Fourier transform  can be defined for  $\mathcal{L}_{\alpha}$, using an exponential operator similarly as was done in \cite{bko}. Then the (sharp) $L^{2}$ Hardy and Hardy-Rellich inequalities for the Grushin operator $\Delta_{\alpha}$, as well as several uncertainty inequalities will follow from this Fourier transform. 
\end{remark}

  Recall that analogues of the classical Kelvin transform have been defined on the Heisenberg group and  Grushin spaces via group theory methods \cite{kor, mor}. Remarkably,  these transforms preserve the class of functions annihilated by the sub-Laplacian and the Grushin operator respectively. Therefore Maxwell's classical construction of harmonic polynomials (see e.g. \cite{mu}) can be extended to these settings \cite{kor}.
The $\mathfrak{sl}_{2}$ structure provides another way, i.e. via differential operators, to construct Grushin-harmonic polynomials.

\begin{theorem} The projection operator, which maps a given homogeneous polynomial of $\delta_{\lambda}$-degree $k$ to its harmonic component of degree $k-2\ell$, is given by
 \begin{eqnarray*}
{\rm Proj}_{\ell}^{k}=\sum_{j=0}^{\lfloor\frac{k}{2}\rfloor -\ell}\alpha_{j}\,\rho^{2j}\,\mathcal{L}_{\alpha}^{j+\ell},
 \end{eqnarray*}
where  $\alpha_{j}=\frac{(-1)^{j}\left(Q/2+k-2\ell-1\right)}{4^{j+\ell} j!\ell!}\frac{\Gamma\left(Q/2+k-2\ell-j-1\right)}{\Gamma\left(Q/2+k-\ell\right)}.$
\end{theorem}
\begin{proof} Let $p_{\ell}$ be a homogeneous polynomial of  $\delta_{\lambda}$-degree $k\in \mathbb{N}_{0}$. Then by the commutator relations in Theorem \ref{sl2}, we have
\begin{eqnarray}\label{p1}
\mathcal{L}_{\alpha}\left(\rho^{k}p_{\ell}\right)=k(k+Q+2\ell-2)\rho^{k-2}p_{\ell}+\rho^{k}\mathcal{L}_{\alpha}p_{\ell}
\end{eqnarray}
where $k,\ell\in \mathbb{N}_{0}$, $k$ being even. Furthermore, if $\Delta_{\alpha}p_{\ell}=0$, Eq.\,\eqref{p1} reduces to
\begin{eqnarray}\label{p11}
\mathcal{L}_{\alpha}\left(\rho^{k}p_{\ell}\right)=k(k+Q+2\ell-2)\rho^{k-2}p_{\ell}.
\end{eqnarray}

By Corollary \ref{fis},  each  homogeneous polynomial $f_{k}$ of $\delta_{\lambda}$-degree $k$ can be decomposed as
 \begin{eqnarray} \label{fs1}
f_{k}=u_{k}+\rho^{2}u_{k-2}+\rho^{4}u_{k-4}+\cdots,
\end{eqnarray}
where each $u_{j}$ is a homogeneous Grushin-harmonic of $\delta_{\lambda}$-degree $j$.
Making use of \eqref{p11} for  several times, we get
 \begin{equation} \label{gn12}
     \begin{split}
&\mathcal{L}_{\alpha}f_{k}=2(Q+2k-4)u_{k-2}+4(Q+2k-6)\rho^{2}u_{k-4}+\cdots,\\
&\mathcal{L}_{\alpha}^{2}f_{k}=8(Q+2k-6)(Q+2k-8)u_{k-4}+\cdots,\\
&\mathcal{L}_{\alpha}^{3}f_{k}=48(Q-2k-8)(Q-2k-10)(Q-2k-12)u_{k-6}+\cdots,
     \end{split}
 \end{equation}
and in general,
 \begin{eqnarray}\label{gn1}
\mathcal{L}_{\alpha}^{v}f_{k}=\sum_{j=v}^{\lfloor\frac{k}{2}\rfloor}
\frac{(2j)!!}{(2j-2v)!!}\frac{(Q+2k-2j-2)!!}{(Q+2k-2j-2v-2)!!}
\rho^{2j-2v}u_{k-2j},
\end{eqnarray}
where
 \begin{eqnarray*}
j!!=\left\{
      \begin{array}{ll}
        j(j-2)(j-4)\cdot 4\times 2, & \hbox{$j$ is even;} \\
        j(j-2)(j-4)\cdot 3 \times1, & \hbox{$j$ is odd.}
      \end{array}
    \right.
\end{eqnarray*}

Eqs.\,\eqref{gn12} or \eqref{gn1} form a set of simultaneous equations. We solve these equations as was done for the ordinary case in \cite[Theorem 5.1.15]{dY} which yields expressions for the homogeneous harmonic polynomials that occur in the decomposition \eqref{fs1}.
\end{proof}



\section{Bounds for the Grushin-harmonic projection}\label{sec4}
 The main purpose of this section is to establish the weighted  $L^{p}-L^{q}$ estimate  for the projection operator $P_{k}$ in \eqref{GH}, which is suitable for deriving  the Carleman estimates.

\subsection{$L^{1}-L^{\infty}$ estimates}
In this subsection, we focus on obtaining  the $L^{1}-L^{\infty}$ estimates for the projection operator  in \eqref{GH}.
To achieve this, we first consider the point-wise bounds for the reproducing kernels of Grushin-harmonics in \eqref{af1}.
Recall that for ordinary spherical harmonics over $\mathbb{S}^{n+m-1}$, there exists a constant $C$ independent of $k$ such that
\begin{eqnarray}\label{rbb1}
\left|K^{n+m}_{k}(\zeta_{1},\zeta_{2})\right|\le C(k+1)^{n+m-2},
\end{eqnarray}
see, e.g.\,\cite{aar}. We show that such an estimate still remains valid  in the (higher step) Grushin setting. It will be proven by extending the approach in \cite{GS}, developed there  for the specific  Grushin operator $\Delta_{1}$ on $\mathbb{R}^{n+1}$. However, it is worth noting that our proof heavily relies on  the addition formula  in Eq.\,\eqref{af}, unavailable around the time \cite{GS} appeared. 

We  give some technical lemmas first.

\begin{lemma} \label{li2} For  $n\ge 4$ and $m\ge 2$,  there exists a positive constant $C$ independent of $k$ and $\phi$ such that
\begin{equation}\label{t1}
    \begin{split}
        &\mathop{\sum_{\widetilde{k}\in\mathbb{N}_{0}}}_{0\le\ell,j\le k}b_{k,\ell,j}^{2}\left[\cos^{j}\phi\sin^{\frac{\ell}{\alpha+1}}\phi
P_{\widetilde{k}}^{(\mu, \gamma-1)}(\cos 2\phi)\right]^{2}\\ &\times(\ell+1)^{n-2}(j+1)^{m-2}
\le C(k+1)^{n+m-2},
    \end{split}
\end{equation}
for every $k\in\mathbb{N}_{0}$, where $\widetilde{k}, \mu, \gamma$ are defined in Theorem \ref{th1} and $b_{k,\ell,j}$ is the normalization constant given in \eqref{nc1}.
\end{lemma}
\begin{proof} {\rm (\uppercase\expandafter{\romannumeral1})} First, we consider the case when $k\equiv 0({\rm  mod}(\alpha+1))$. In this case by the condition \[\widetilde{k}=\frac{k-\ell-j(\alpha+1)}{2(\alpha+1)}\in \mathbb{N}_{0}, \] we must have $\ell\equiv 0({\rm  mod}(\alpha+1))$. Thus we can denote $\ell=(\alpha+1)\ell_{0}$ and $k=(\alpha+1)k_{0}$ with $\ell_{0}, k_{0}\in \mathbb{N}_{0}$. Now the left hand side of Eq.\,\eqref{t1} becomes
\begin{equation}\label{i11}
    \begin{split}
       &\sum_{i=0}^{\left\lfloor k_{0}/2\right\rfloor}\sum_{j+\ell_{0}=k_{0}-2i} b_{k,\ell,j}^{2} \left(\cos^{j}\phi\sin^{\ell_{0}}\phi  P_{\frac{k_{0}-\ell_{0}-j}{2}}^{\left(\frac{n-2}{2(\alpha+1)}+\ell_{0}, \frac{m}{2}+j-1\right)}(\cos 2\phi)\right)^{2}\\&\times((\alpha+1)\ell_{0}+1)^{n-2}(j+1)^{m-2}. 
    \end{split}
\end{equation}

On the other side,  the addition formula \eqref{af} can be equivalently written as
\begin{equation}\label{add1}
    \begin{split}
        &e_{u+v,{\rm k}}\,C_{{\rm k}}^{\left(\frac{u+v}{2}-1\right)}(\cos \phi\cos \xi\cos\theta_{1}+\sin\phi\sin \xi\cos\theta_{2} )
\\=&\sum_{i=0}^{\lfloor {\rm k}/2\rfloor}\sum_{j+\ell={\rm k}-2i}f_{{\rm k},\ell,j}(\cos 2\phi)f_{{\rm k},\ell,j}(\cos 2\xi)\\&\times e_{u,j}\,C_{j}^{(\frac{u}{2}-1)}(\cos\theta_{1})\,e_{v,\ell}\,C_{\ell}^{(\frac{v}{2}-1)}(\cos\theta_{2}),
    \end{split}
\end{equation}
where $f_{{\rm k},\ell,j}$ is given in Theorem \ref{adfi}, and
\[ e_{u, j}:=\frac{1}{2}\Gamma\left(\frac{u}{2}-1\right)\left(\frac{u}{2}-1+j\right).\] Note that we used the sub-index notation ${\rm k}$ in \eqref{add1} on purpose, which is not the homogeneous degree $k$ now. It will be set by different values related with the homogeneous degree $k$ in the subsequent proof.

Now plugging $u=m$, $v=\frac{n-2}{\alpha+1}+2$, ${\rm k}=k_{0}$, $\xi=\phi$ and $\theta_{1}=\theta_{2}=0$ in  \eqref{add1}, we obtain
\begin{equation}\label{px1}
    \begin{split}
      &e_{u+v,k_{0}}C_{k_{0}}^{\left(\frac{u+v}{2}-1\right)}(1)
 =\sum_{i=0}^{\left\lfloor k_{0}/2\right\rfloor}\sum_{j+\ell_{0}=k_{0}-2i} e_{u,j}C_{j}^{(\frac{u}{2}-1)}(1)\,e_{v,\ell_{0}}C_{\ell_{0}}^{(\frac{v}{2}-1)}(1)\\
 &\times b_{k,\ell,j}^{2} 
 \left(\cos^{j}\phi\,\sin^{\ell_{0}}\phi\,  P_{\frac{k_{0}-\ell_{0}-j}{2}}^{\left(\frac{n-2}{2(\alpha+1)}+\ell_{0}, \frac{m}{2}+j-1\right)}(\cos 2\phi)\right)^{2}.
    \end{split}
\end{equation}

Comparing \eqref{i11} with the right hand side of \eqref{px1}, together with 
\begin{equation}\label{sim1}
    \begin{split}
     e_{u,k}C_{k}^{(\frac{u}{2}-1)}(1)=&\frac{\left(u/2-1+k\right)\Gamma\left(u/2-1\right)\Gamma(k+u-2)}{\Gamma(u-2)\Gamma(k+1)}\\
\sim& (k+1)^{u-2},  
    \end{split}
\end{equation}
it follows that  \eqref{i11} is bounded by
\begin{eqnarray*}
 &&C(k+1)^{n-2-\frac{n-2}{\alpha+1}}\sum_{i=0}^{\left\lfloor k_{0}/2\right\rfloor}\sum_{j+\ell_{0}=k_{0}-2i} (j+1)^{m-2} (\ell_{0}+1)^{\frac{n-2}{\alpha+1}}\\
 &&\times \,b_{k,\ell,j}^{2} \left(\cos^{j}\phi\,\sin^{\ell_{0}}\phi\,  P_{\frac{k_{0}-\ell_{0}-j}{2}}^{\left(\frac{n-2}{2(\alpha+1)}+\ell_{0}, \frac{m}{2}+j-1\right)}(\cos 2\phi)\right)^{2}\\ 
&\le& C(k+1)^{n-2-\frac{n-2}{\alpha+1}}\left[ e_{u+v,k_{0}}C_{k_{0}}^{\left(\frac{u+v}{2}-1\right)}(1)\right]\\&\le&
C(k+1)^{n-2-\frac{n-2}{\alpha+1}}\cdot (k+1)^{\frac{n-2}{\alpha+1}+m}\\
&\le& C(k+1)^{n+m-2},
\end{eqnarray*}
where in the second step we have used  $n-2\ge \frac{n-2}{\alpha+1}$ when $n\ge 2$. This proves the lemma when $k\equiv 0({\rm  mod}(\alpha+1))$.

{\rm (\uppercase\expandafter{\romannumeral2})} For the cases  when \[k\equiv k_{\alpha}({\rm  mod}(\alpha+1)), \quad  k_{\alpha}\in \{1, \ldots, \alpha\},\] or equivalently $k=k_{0}(\alpha+1)+k_{\alpha}$,  we  have \[\ell=(\alpha+1)\ell_{0}+k_{\alpha},\] by the condition $\widetilde{k}=\frac{k-\ell-j(\alpha+1)}{2(\alpha+1)}\in \mathbb{N}_{0}$.

Thus in these cases, the left hand side of \eqref{t1} becomes
\begin{equation}\label{i2}
    \begin{split}
      & \sin^{\frac{2k_{\alpha}}{\alpha+1}}\phi \sum_{i=0}^{\left\lfloor k_{0}/2\right\rfloor}\sum_{j+\ell_{0}=k_{0}-2i} \left(\cos^{j}\phi\, \sin^{\ell_{0}}\phi \, P_{\frac{k_{0}-\ell_{0}-j}{2}}^{\left(\frac{n-2+2k_{\alpha}}{2(\alpha+1)}+\ell_{0}, \frac{m}{2}+j-1\right)}(\cos 2\phi)\right)^{2} \\& \times b_{k,\ell,j}^{2} (\ell+1)^{n-2}(j+1)^{m-2}.
    \end{split}
\end{equation}

Now setting $u=m$, $v=\frac{n-2+2k_{\alpha}}{\alpha+1}+2$, ${\rm k}=k_{0}$ and $\phi=\xi$ in the addition formula \eqref{add1}, we get
\begin{eqnarray}\label{px2}
         &&e_{u+v,k_{0}}C_{k_{0}}^{\left(\frac{m}{2}+\frac{n-2+2k_{\alpha}}{2(\alpha+1)}\right)}(\cos^{2} \phi\cos\theta_{1}+\sin^{2}\phi\cos\theta_{2})\\
 &=&\sum_{i=0}^{\left\lfloor k_{0}/2\right\rfloor}\sum_{j+\ell_{0}=k_{0}-2i} e_{u,j}C_{j}^{\left(\frac{m}{2}-1\right)}(\cos\theta_{1})\,e_{v,\ell_{0}}C_{\ell_{0}}^{\left(\frac{n-2+2k_{\alpha}}{2(\alpha+1)}\right)}(\cos \theta_{2})\nonumber\\
 &&\times b_{k,\ell,j}^{2} \left(\cos^{j}\phi\,\sin^{\ell_{0}}\phi\,  P_{\frac{k_{0}-\ell_{0}-j}{2}}^{\left(\frac{n-2+2k_{\alpha}}{2(\alpha+1)}+\ell_{0}, \frac{m}{2}+j-1\right)}(\cos 2\phi)\right)^{2}.\nonumber
\end{eqnarray}

When $n\ge 4,$ it is easy to see that  \[n-2\ge \frac{n-2+2k_{\alpha}}{\alpha+1}.\]  Similarly to the previous case, set $\theta_{1}=\theta_{2}=0$ in \eqref{px2}. Then it follows that \eqref{i2} is bounded by
\begin{eqnarray*}
&&C\left(\sin^{\frac{2k_{\alpha}}{\alpha+1}}\phi\right)(k+1)^{n-2-\frac{n-2+2k_{\alpha}}{\alpha+1}}\cdot e_{\frac{n-2+2k_{\alpha}}{\alpha+1}+m+2, k_{0}} C_{k_{0}}^{\left(\frac{m}{2}+\frac{n-2+2k_{\alpha}}{2(\alpha+1)}\right)}(1 )
\\&& \le  C(k+1)^{n+m-2},
\end{eqnarray*}
where we have used \eqref{sim1} again. This completes the proof.
\end{proof}
\begin{remark} \label{rb23} When $n=2$ or $3$, the above proof in {\rm (\uppercase\expandafter{\romannumeral2})} shows that the left-hand side of Eq.\,\eqref{t1} can be  bounded by  $C(k+1)^{m+2}$ for $\alpha\in \mathbb{N} 
$, which is not as tight as the desired $C(k+1)^{n+m-2}$. Subtle estimates seem necessary to achieve that. However,  the less tight bound is  already sufficient for our later purpose. 
\end{remark}

The case when $\alpha=1$ is much easier. The estimate in Eq.\,\eqref{t1} remains valid when $n\in \{2, 3\}$ and $m\ge 2$. It can be proved by utilizing  the addition formula \eqref{af} and  employing subtle estimates of Gegenbauer polynomials through an extension of the approach in \cite[p.\,143-146]{GS}. As the integral estimations are similar with \cite{GS}, we only list the result here and 
omit the proof. 

\begin{lemma} \label{ggg1} When $n=2$ or $3$ and $\alpha=1$, there exists a positive constant $C$ independent of $k$ and $\phi$ such that
\begin{eqnarray*}
&&\mathop{\sum_{\widetilde{k}\in\mathbb{N}_{0}}}_{0\le\ell,j\le k}b_{k,\ell,j}^{2}\left[\cos^{j}\phi\,\sin^{\frac{\ell}{\alpha+1}}\phi
\,P_{\widetilde{k}}^{(\mu, \gamma-1)}(\cos 2\phi)\right]^{2}\\ &&\times(\ell+1)^{n-2}(j+1)^{m-2}
\le C(k+1)^{n+m-2}\nonumber,
\end{eqnarray*}
for every $k\in\mathbb{N}_{0}$, where $\widetilde{k}, \mu, \gamma$ are defined in Theorem \ref{th1} and $b_{k,\ell,j}$ is the normalization constant given in \eqref{nc1}.
\end{lemma}

Now,  the $L^{1}-L^{\infty}$ estimates for the projection operator $P_{k}$ in \eqref{GH} can be readily derived from the pointwise bounds of the series on the left-hand side of \eqref{t1}.
\begin{theorem}\label{wl1} When {\rm (1)} $n\ge 4$, $m\ge 2$ and $\alpha\in \mathbb{N}$, or {\rm (2)} $n\in \{2, 3\}$, $m\ge 2$ and $\alpha=1$,
there exists $C>0$ independent of $k$  such that
\begin{eqnarray*}
\|P_{k}(g)\|_{L^{\infty}(\Omega, {\rm d}\Omega)}\le C(k+1)^{n+m-2}\|g\|_{L^{1}(\Omega, {\rm d}\Omega)}
\end{eqnarray*}
for every $g\in L^{1}(\Omega, {\rm d}\Omega)$.
\end{theorem}
\begin{proof} By the integral expression \eqref{int1} of the projection operator $P_{k}$, it suffices to show that
\begin{eqnarray}\label{d1}|G_{k}(\phi,\omega, \xi, \eta)|\le C(k+1)^{n+m-2}\end{eqnarray}
for $0\le \phi, \xi\le \pi/2$ and $\omega,\eta\in \mathbb{S}^{n-1}\times\mathbb{S}^{m-1}$.

On the other side, Lemmas \ref{li2} and \ref{ggg1} together with \eqref{rbb1} yield
\begin{equation}\label{d2}
    \begin{split}
       |G_{k}(\phi,\omega, \phi, \omega)| &\le C\mathop{\sum_{\widetilde{k}\in\mathbb{N}_{0}}}_{0\le\ell,j\le k}b_{k,\ell,j}^{2}h_{k,\ell,j}^{2}(\phi)(\ell+1)^{n-2}(j+1)^{m-2}\\ &\le C(k+1)^{n+m-2}.
    \end{split}
\end{equation}
Then by the explicit expression \eqref{af1} for $G_{k}(\phi,\omega, \xi, \eta)$ and the Cauchy-Schwarz inequality, we have
\begin{equation}\label{dd3}
 |G_{k}(\phi,\omega, \xi, \eta)|\le \left|G_{k}(\phi,\omega, \phi, \omega)\right|^{\frac{1}{2}} \left|G_{k}(\xi, \eta, \xi, \eta)\right|^{\frac{1}{2}}. 
\end{equation}
Combining  \eqref{d2} and \eqref{dd3}, we obtain \eqref{d1}. This concludes the proof.
\end{proof}
\begin{remark} \label{s1et} By formally setting $m=1$, we get the  $L^{1}-L^{\infty}$ estimates for the projectors associated to the Grushin operator $\Delta_{\alpha}$ defined on $\mathbb{R}^{n+1}$. It further reduces to the estimate for $\Delta_{1}$  on $\mathbb{R}^{n+1}$ in \cite[Theorem 3.1]{GS}.
\end{remark}
\begin{remark} By Remark \ref{rb23} and the proof of Theorem \ref{wl1}, when $n\in \{2, 3\}$  and  $\alpha\in \mathbb{N} \backslash \{1\}$, there exists a constant $C>0$ independent of $k$  such that
\begin{eqnarray*}
\|P_{k}(g)\|_{L^{\infty}(\Omega, {\rm d}\Omega)}\le C(k+1)^{m+2}\|g\|_{L^{1}(\Omega, {\rm d}\Omega)}
\end{eqnarray*}
for every $g\in L^{1}(\Omega, {\rm d}\Omega)$.
\end{remark}
As a summary, we put the $L^{1}-L^{\infty}$ bounds for the projection operator $P_{k}$ in \eqref{GH} in Table \ref{tab:example}.
\begin{table}
        \centering
         \caption{$L^{1}-L^{\infty}$ bounds for the projection operator $P_{k}$}
         \label{tab:example}
\begin{tabular}{cccc}
   \toprule
     $n$  & $m$   & $\alpha=1$  & $\alpha\in \mathbb{N} \backslash \{1\}$  \\
   \midrule
    $n\ge 4$ & $m\ge 1$ & $C(k+1)^{n+m-2}$& $C(k+1)^{n+m-2}$ \\
   $n\in \{2,3\}$ & $ m\ge 1$ & $C(k+1)^{n+m-2}$ & $C(k+1)^{m+2}$ \\
   \bottomrule
\end{tabular}
\end{table}

\subsection{Weighted $L^{2}-L^{2}$ estimates} \label{421}


This subsection is  devoted to deriving the weighted $L^{2}-L^{2}$ estimates for the projector $P_{k}$ in \eqref{int1}. The crucial ingredient is the weighted $L^{2}$-norm estimate for Jacobi polynomials. Here we  provide two different approaches. The first one is based on the Bernstein-type inequality  \eqref{bs} for Jacobi polynomials obtained in \cite{HS}. The second one is based on the connection formula and some explicit calculations. Our proofs are short, straightforward,  and completely different from the approach developed in \cite[Lemma 3.3]{GS} for $\Delta_{1}$ on $\mathbb{R}^{n+1}$. 

Recall that the angle function in \eqref{ps1} is defined by \begin{eqnarray*} \psi(\phi):=\sin^{\frac{2\alpha}{\alpha+1}}\phi=\left(r_{1}/\rho\right)^{2\alpha}.\end{eqnarray*}
\begin{lemma}\label{l2w} Suppose $\alpha \in \mathbb{N}$, $n, m\ge 2$ and $\displaystyle 0\le\beta<(\alpha+1)/(4\alpha)$,  then there exists  $C>0$ depending only on $\alpha, \beta, n, $ and $m$, such that, for every $g(\phi, \omega_{1}, \omega_{2})\in \mathcal{H}_{k}^{\alpha}$, we have
\begin{eqnarray*}
\left\|\psi^{-\beta}(\phi) g\right\|_{L^{2}(\Omega, {\rm d}\Omega)}\le C(k+1)^{\frac{1}{4}}\|g\|_{L^{2}(\Omega, {\rm d}\Omega)}.
\end{eqnarray*}
\end{lemma}
\begin{proof} For a fixed integer $k\ge0$,  let
\begin{equation}\label{o1}
 \begin{split}
  &\bigg\{b_{k,\ell,j}h_{k,\ell, j}(\phi)Y_{\ell, p}(\omega_{1})Y_{j, q}(\omega_{2})\bigg|1\le p\le d_{\ell}(n), 1\le q\le d_{j}(m), \\ & 0\le \ell, j\le k, \, {\rm and} \,\, \ell+(\alpha+1)j\equiv k\left({\rm mod}(2\alpha+2)\right)\bigg\}  
\end{split}   
\end{equation}
be an orthonormal basis for $\mathcal{H}_{k}^{\alpha}\subset L^{2}(\Omega, {\rm d}\Omega)$. Note that \eqref{o1} is also an orthogonal set in $L^{2}(\Omega, \psi^{-2\beta}(\phi) {\rm d}\Omega)$. 
Thus, it is sufficient to show that there exists a constant $C>0$ independent of $k$ such that
\begin{equation}\label{g2}
    \begin{split}
        &\int_{0}^{\frac{\pi}{2}}|b_{k,\ell,j}h_{k,\ell, j}(\phi)|^{2}\left(\sin\phi\right)^{\frac{n-1+\alpha-4\alpha\beta}{\alpha+1}}\left(\cos\phi\right)^{m-1} {\rm d}\phi\\
=&\,\int_{0}^{\frac{\pi}{2}}b_{k,\ell,j}^{2}
\left(P_{\frac{k-\ell-j(\alpha+1)}{2(\alpha+1)}}^{\left(\frac{n-2+2\ell}{2(\alpha+1)}, j+\frac{m}{2}-1\right)}(\cos 2\phi)\right)^{2}\\&\times\left(\sin\phi\right)^{\frac{n-1+\alpha+2\ell-4\alpha\beta}{\alpha+1}}\left(\cos\phi\right)^{m-1+2j} {\rm d}\phi
\\ \le&\, C(k+1)^{\frac{1}{2}},
    \end{split}
\end{equation}
for  $\displaystyle 0\le\beta<(\alpha+1)/(4\alpha)$. 

On the other hand, we have
\begin{eqnarray}\label{is1}\\
      &&\int_{0}^{\frac{\pi}{2}}b_{k,\ell,j}^{2}
\left(P_{\frac{k-\ell-j(\alpha+1)}{2(\alpha+1)}}^{\left(\frac{n-2+2\ell}{2(\alpha+1)}, j+\frac{m}{2}-1\right)}(\cos 2\phi)\right)^{2}\left(\sin\phi\right)^{\frac{n-1+\alpha+2\ell-4\alpha\beta}{\alpha+1}}\left(\cos\phi\right)^{m-1+2j} {\rm d}\phi \nonumber\\
&=& \frac{1}{2}\int_{-1}^{1}b_{k,\ell,j}^{2}
\left|P_{\frac{k-\ell-j(\alpha+1)}{2(\alpha+1)}}^{\left(\frac{n-2+2\ell}{2(\alpha+1)}, j+\frac{m}{2}-1\right)}(x)\right|^{2}\left(\frac{1-x}{2}\right)^{\frac{n-2+2\ell-4\alpha\beta}{2(\alpha+1)}}\left(\frac{1+x}{2}\right)^{\frac{m}{2}-1+j} {\rm d}x \nonumber
\\&=& \frac{\sqrt{2}}{4}\left(\frac{k-\ell-j(\alpha+1)}{\alpha+1}+\frac{n-2+2\ell}{2(\alpha+1)}+j+\frac{m}{2}\right) \nonumber
\\ &&\times \int_{-1}^{1}\left|(1-x^{2})^{\frac{1}{4}}
g_{\frac{k-\ell-j(\alpha+1)}{2(\alpha+1)}}^{\left(\frac{n-2+2\ell}{2(\alpha+1)}, j+\frac{m}{2}-1\right)}(x)\right|^{2}\left(\frac{1-x}{2}\right)^{-\frac{2\alpha\beta}{\alpha+1}-\frac{1}{2}}\left(\frac{1+x}{2}\right)^{-\frac{1}{2}}{\rm d}x \nonumber\\
&\le&  C(k+1)^{\frac{1}{2}}\int_{-1}^{1}\left(\frac{1-x}{2}\right)^{-\frac{2\alpha\beta}{\alpha+1}-\frac{1}{2}}\left(\frac{1+x}{2}\right)^{-\frac{1}{2}}{\rm d}x \nonumber\\
&\le & C(k+1)^{\frac{1}{2}}, \nonumber
\end{eqnarray}
where $g(x)$ is defined in \eqref{g32} and in the third step we used the  the Bernstein-type estimate \eqref{bs}. 
The last step  is  guaranteed by the condition $0\le \beta<(\alpha+1)/(4\alpha)$.
This completes the proof.
\end{proof}
\begin{remark} It is seen that if $\alpha\in \mathbb{N}$, then $1/2\le (\alpha+1)/(4\alpha)$ if and only if $\alpha=1$. We will show in the subsequent sections that the estimate in the above Lemma \ref{l2w} is already enough to completely remove the degenerate weight $\psi$ in \eqref{pot1} for the cases where $\alpha=1$.  It will lead to interesting applications in the analysis of $H$-type groups.
\end{remark}

Now, we start the second approach for the weighted  $L^{2}-L^{2}$ estimate for general  $\alpha\in \mathbb{N}$. Although the following explicit formulas on weighted $L^{2}$-norms of Jacobi polynomials are already known by experts, we provide a brief proof for the sake of completeness and to ensure self-containment.
\begin{lemma} Assume $\gamma, \alpha, \beta>-1$. Then we have
\begin{equation}\label{id1}
    \begin{split}
        I_{n}^{(\gamma, \alpha; \beta)}:=&\int_{0}^{\pi/2}\left(P_{n}^{(\gamma,\beta)}(\cos 2\theta)\right)^{2}(\sin \theta)^{2\alpha+1}\,(\cos \theta)^{2\beta+1}\, {\rm d}\theta\\
        =& \left(\frac{(\beta+1)_{n}}{(\alpha+\beta+2)_{n}}\right)^{2}\sum_{k=0}^{n}\left(\frac{\alpha+\beta+2k+1}{\alpha+\beta+1}\right.\\
   &\times \left.\frac{(\gamma-\alpha)_{n-k}(\alpha+\beta+1)_{k}(\beta+\gamma+n+1)_{k}}{(n-k)!(\beta+1)_{k}(\alpha+\beta+n+2)_{k}}\right)^{2}\\
   &\times \frac{1}{2(2k+\alpha+\beta+1)}\frac{\Gamma(k+\alpha+1)\Gamma(k+\beta+1)}{\Gamma(k+\alpha+\beta+1)k!}.
    \end{split}
\end{equation} 
In particular when $\gamma>0$ and $\beta>-1$,  it holds that
\begin{equation}\label{a-1}
    \begin{split}
        I_{n}^{(\gamma, \gamma-1; \beta)}=&\, \int_{0}^{\pi/2}\left(P_{n}^{(\gamma,\beta)}(\cos 2\theta)\right)^{2}(\sin \theta)^{2\gamma-1}(\cos \theta)^{2\beta+1}\,{\rm d}\theta\\
        =&\, \frac{\Gamma(\gamma+n+1)\Gamma(\beta+n+1)}{2\gamma\Gamma(\gamma+\beta+n+1)n!}.
    \end{split}
\end{equation} 
\end{lemma}
\begin{proof}  By the connection formula for Jacobi polynomials in \cite[Theorem 7.1.3]{aar}, i.e.
\begin{equation*}
\begin{split}
    P_{n}^{(\gamma, \beta)}(x)=&\frac{(\beta+1)_{n}}{(\alpha+\beta+2)_{n}}\sum_{k=0}^{n}\frac{\alpha+\beta+2k+1}{\alpha+\beta+1}\\
   &\times \frac{(\gamma-\alpha)_{n-k}(\alpha+\beta+1)_{k}(\beta+\gamma+n+1)_{k}}{(n-k)!(\beta+1)_{k}(\alpha+\beta+n+2)_{k}}
   P_{k}^{(\alpha, \beta)}(x),
\end{split}    
\end{equation*}  
and the orthogonal condition \eqref{or1} (see also \eqref{bab}), we obtain the first explicit formula \eqref{id1}. The other formula for $ I_{n}^{(\gamma, \gamma-1; \beta)}$  is obtained simply by changing variables in \cite[\S 16.4, (6)]{bate}.
\end{proof}
The following bound  is important for the second estimate.
\begin{lemma} \label{kl1} Let $\gamma>0,\beta>-1$ and  $\gamma-1<\alpha<\gamma$. Then we have 
\begin{equation*}
     \frac{I_{n}^{(\gamma, \alpha; \beta)}}{B_{n}^{\gamma, \beta}}\le \left(\frac{1}{\gamma}(2n+\gamma+\beta+1)\right)^{\gamma-\alpha},
\end{equation*}
where $B_{n}^{\gamma, \beta}$ and $I_{n}^{(\gamma, \alpha; \beta)}$  are defined in \eqref{bab} and \eqref{id1}, respectively.
\end{lemma}
\begin{proof} Denote $ {\rm d}\mu(\theta)=\left(P_{n}^{(\gamma,\beta)}(\cos 2\theta)\right)^{2}(\sin \theta)^{2\gamma+1}(\cos \theta)^{2\beta+1}/B_{n}^{\gamma, \beta} \,{\rm d}\theta$. Obviously we have $\int_{0}^{\pi/2} {\rm d}\mu(\theta)
=1$. 

By H\"older's inequality, we get
    \begin{equation*}
    \begin{split}
        \frac{I_{n}^{(\gamma, \alpha; \beta)}}{B_{n}^{\gamma, \beta}} & = \frac{1}{B_{n}^{\gamma, \beta}}\int_{0}^{\pi/2} \left(P_{n}^{(\gamma,\beta)}(\cos 2\theta)\right)^{2}(\sin \theta)^{2\alpha+1}(\cos \theta)^{2\beta+1}\,{\rm d}\theta\\
        & = \int_{0}^{\pi/2}1\cdot\frac{1}{(\sin \theta)^{2(\gamma-\alpha)}}\,{\rm d}\mu(\theta)\\
        & \le \left(\int_{0}^{\pi/2}\frac{1}{(\sin \theta)^{2}}\,{\rm d}\mu(\theta)\right)^{\gamma-\alpha}\\
        & =\left( \frac{I_{n}^{(\gamma, \gamma-1; \beta)}}{B_{n}^{\gamma, \beta}}\right)^{\gamma-\alpha}\\
        & = \left((2n+\gamma+\beta+1)/\gamma \right)^{\gamma-\alpha}.
    \end{split}
\end{equation*} 
The last step follows from the explicit constants in Eqs.\,\eqref{bab} and \eqref{a-1}.
\end{proof}
\begin{remark}
By Hilb's approximation for Jacobi polynomials (see e.g. \cite[Theorem 8.21.12]{or}), 
the asymptotic behaviour of $I_{n}^{(\gamma, \alpha; \beta)}$ can be obtained  (see \cite[p.1064]{adgp}),
 \begin{equation*}
\begin{split}
   I_{n}^{(\gamma, \alpha; \beta)}=\binom{n+\gamma}{n}^{2}\frac{\Gamma(2(\gamma-\alpha)-1)\Gamma(1+\alpha)}{2\Gamma(\gamma-\alpha)^{2}\Gamma(2\gamma-\alpha)}N^{-2(\alpha+1)}+\mathcal{O}(n^{\gamma-\alpha-\frac{3}{2}}), 
\end{split}         
\end{equation*}
where $N=n+(\gamma+\beta+1)/2$. It suggests that the optimal estimate  for $\gamma-1<\alpha<\gamma-1/2$ in Lemma \ref{kl1}  should be 
\begin{equation*}
     \frac{I_{n}^{(\gamma, \alpha; \beta)}}{B_{n}^{\gamma, \beta}}\le C \left(2n+\gamma+\beta+1\right)^{2(\gamma-\alpha)-1},
\end{equation*}
with a constant $C$ only depending  on $\gamma-\alpha$. Unfortunately, it is currently unclear to the authors how to prove this.  
\end{remark}
\begin{remark} For $\alpha, \beta>0$, the following relations hold 
\begin{equation*}
    \frac{I_{n-1}^{(\gamma+1, \alpha+1; \beta+1)}}{B_{n-1}^{\gamma+1, \beta+1}}\le \frac{I_{n}^{(\gamma, \alpha; \beta)}}
    {B_{n}^{\gamma, \beta}}, \qquad \frac{I_{n-1}^{(\gamma+2, \alpha+2; \beta)}}{{B_{n-1}^{\gamma+2, \beta}}}\le \frac{I_{n}^{(\gamma, \alpha; \beta)}}{B_{n}^{\gamma, \beta}}.
\end{equation*}
The proof is similar as the subsequent \eqref{rat1}, hence we omit it.    
\end{remark}

The cases with $\gamma=0$ are not included  in Lemma \ref{kl1}, because $I_{n}^{(0, -1; \beta)}$ may not converge. Thanks to the bound given in \eqref{bou1}, we still have bounds for the following  cases.
\begin{lemma} \label{dxsa1}  Suppose $\gamma,\beta\in \mathbb{N}_{0}$ and  $\gamma-1<\alpha<\gamma$. For any given $0<\varepsilon<1$, 
 there exists $C>0$ only depending on $\varepsilon$ and $\gamma-\alpha$  such that 
\begin{equation*}
     \frac{I_{n}^{(\gamma, \alpha; \beta)}}{B_{n}^{\gamma, \beta}}\le C \left(2n+\gamma+\beta+1\right)^{\frac{\gamma-\alpha}{1-\varepsilon}},
\end{equation*}
where $B_{n}^{\gamma, \beta}$ and $I_{n}^{(\gamma, \alpha; \beta)}$  are defined in \eqref{bab} and \eqref{id1}, respectively.
\end{lemma}
\begin{proof} 
By \eqref{bou1}, we have  \begin{equation}\label{b1c1}
    \frac{1}{B_{n}^{\gamma, \beta}}  \left(P_{n}^{(\gamma,\beta)}(\cos 2\theta)\right)^{2}(\sin \theta)^{2\gamma}(\cos \theta)^{2\beta} \le 2(2n+\gamma+\beta+1),
\end{equation}
for all $\gamma,\beta\in \mathbb{N}_{0}$ and $\theta\in [0,\pi/2]$.

Similarly as was done in Lemma \ref{kl1}, by H\"older inequality, we have
\begin{equation*}
    \begin{split}
        \frac{I_{n}^{(\gamma, \alpha; \beta)}}{B_{n}^{\gamma, \beta}} &= \frac{1}{B_{n}^{\gamma, \beta}}\int_{0}^{\pi/2} \left(P_{n}^{(\gamma,\beta)}(\cos 2\theta)\right)^{2}(\sin \theta)^{2\alpha+1}(\cos \theta)^{2\beta+1}\,{\rm d}\theta\\
        & = \int_{0}^{\pi/2}1\cdot \frac{1}{(\sin \theta)^{2(\gamma-\alpha)}}\, {\rm d}\mu(\theta)\\
        & \le \left(\int_{0}^{\pi/2}\frac{1}{(\sin \theta)^{2(1-\varepsilon)}}\, {\rm d}\mu(\theta)\right)^{\frac{\gamma-\alpha}{1-\varepsilon}}\\
        & \le  \left(2(2n+\gamma+\beta+1)\int_{0}^{\pi/2}\frac{\sin \theta\cos\theta}{(\sin \theta)^{2(1-\varepsilon)}} \,{\rm d}\theta\right)^{\frac{\gamma-\alpha}{1-\varepsilon}}\\
        & \le  C \left(2n+\gamma+\beta+1\right)^{\frac{\gamma-\alpha}{1-\varepsilon}},      
    \end{split} 
\end{equation*}
where we have used \eqref{b1c1} in the fourth step.
\end{proof}

Now, we can go back to the Grushin setting.
\begin{lemma} \label{nwf1} Assume $\displaystyle 0\le\beta<1/2$. When $n\geq 3$ and $m\ge 2$, there exists  $C>0$ depending only on $\alpha, \beta, n, $ and $m$, such that, for every $g\in \mathcal{H}_{k}^{\alpha}$,
\begin{eqnarray*}
\left\|\psi^{-\beta}(\phi) g\right\|_{L^{2}(\Omega, {\rm d}\Omega)}\le C(k+1)^{\frac{\alpha}{2(\alpha+1)}}\|g\|_{L^{2}(\Omega, {\rm d}\Omega)}.
\end{eqnarray*}
When  $n=2$ and $m\ge 2$ is even, for a given $0<\varepsilon<1$, there exists  $C>0$ depending only on $\varepsilon, \alpha, \beta, n, $ and $m$, such that, for every $g\in \mathcal{H}_{k}^{\alpha}$,
\begin{eqnarray*}
\left\|\psi^{-\beta}(\phi) g\right\|_{L^{2}(\Omega, {\rm d}\Omega)}\le C(k+1)^{\frac{\alpha}{2(\alpha+1)(1-\varepsilon)}}\|g\|_{L^{2}(\Omega, {\rm d}\Omega)}.
\end{eqnarray*}
\end{lemma}
\begin{proof} (1) From the proof of Lemma \ref{l2w}, it is sufficient to show that
\begin{equation*}\label{tp1}
    \begin{split}
        &\int_{0}^{\frac{\pi}{2}}|b_{k,\ell,j}h_{k,\ell, j}(\phi)|^{2}\left(\sin\phi\right)^{\frac{n-1+\alpha-4\alpha\beta}{\alpha+1}}\left(\cos\phi\right)^{m-1} {\rm d}\phi\\
=&\, \int_{0}^{\frac{\pi}{2}}b_{k,\ell,j}^{2}
\left(P_{\frac{k-\ell-j(\alpha+1)}{2(\alpha+1)}}^{\left(\frac{n-2+2\ell}{2(\alpha+1)}, j+\frac{m}{2}-1\right)}(\cos 2\phi)\right)^{2}\\&\times\left(\sin\phi\right)^{\frac{n-1+\alpha+2\ell-4\alpha\beta}{\alpha+1}}\left(\cos\phi\right)^{m-1+2j}\, {\rm d}\phi
\\ \le& \,C(k+1)^{\frac{\alpha}{\alpha+1}},
    \end{split}
\end{equation*}
for $\displaystyle 0\le\beta<1/2$. When $n>2$ and $m\ge 2$, it follows from Lemma \ref{kl1} and the fact $\frac{n-2+2\ell}{2(\alpha+1)}\ge \frac{1}{2(\alpha+1)}$.

(2) When $n=2$ and $m\ge 2$ even, from the proof of the first part, it is sufficient to consider the cases when the first parameter of the Jacobi polynomials is $0$. In this case, we have
\begin{equation*}
    \begin{split}
        &\int_{0}^{\frac{\pi}{2}}|b_{k,0,j}h_{k,0, j}(\phi)|^{2}\left(\sin\phi\right)^{\frac{1+\alpha-4\alpha\beta}{\alpha+1}}\left(\cos\phi\right)^{m-1} {\rm d}\phi\\
=&\,\int_{0}^{\frac{\pi}{2}}b_{k,0,j}^{2}
\left(P_{\frac{k-j(\alpha+1)}{2(\alpha+1)}}^{\left(0,  j+\frac{m}{2}-1\right)}(\cos 2\phi)\right)^{2}\\&\times\left(\sin\phi\right)^{\frac{1+\alpha+2\ell-4\alpha\beta}{\alpha+1}}\left(\cos\phi\right)^{m-1+2j} {\rm d}\phi
\\ \le&\, C(k+1)^{\frac{\alpha}{(\alpha+1)(1-\varepsilon)}},
    \end{split}
\end{equation*}
for $\displaystyle 0\le\beta<1/2$. The last step follows from Lemma \ref{dxsa1}.
\end{proof}

Following \cite[Theorem 4.1]{GS}, the weighted $L^{2}-L^{2}$ estimate for the projector $P_{k}$ in \eqref{int1} is derived from Lemma \ref{nwf1} (for $n\ge 3$) and Lemma \ref{l2w} (for $n=2$ and $\alpha=1$).
\begin{theorem}\label{wl2} Suppose $\displaystyle 0\le\beta<1/2$. When {\rm (1)} $n\ge 3$, $m\ge 2$ and $\alpha\in \mathbb{N}$, or {\rm (2)} $n=2$, $m\ge 2$ and $\alpha=1$,  there exists  a constant $C>0$, depending  on $\alpha, \beta, n$ and $m$, such that for every $g\in L^{2}(\Omega, {\rm d}\Omega)$, the following holds
\begin{eqnarray}\label{l2w1}
\int_{\Omega}\left|\psi^{-\beta}(\phi) P_{k}\left(\psi^{-\beta}(\cdot)g\right)(\phi, \omega_{1}, \omega_{2})\right|^{2}{\rm d}\Omega\le C(k+1)^{\frac{2\alpha}{\alpha+1}}\int_{\Omega}|g|^{2}{\rm d}\Omega.
\end{eqnarray}
\end{theorem}
\begin{proof}
Let
\begin{eqnarray}
T_{\beta, k}(g)(\phi, \omega_{1}, \omega_{2})=\psi^{-\beta}(\phi) P_{k}(g)(\phi, \omega_{1}, \omega_{2}).
\end{eqnarray}
Then by Lemma \ref{l2w} and Lemma \ref{nwf1}, it holds
\begin{equation} \label{l2w2}
    \begin{split}
        \|T_{\beta, k}(g)\|_{L^{2}(\Omega, {\rm d}\Omega)}\le&\, C(k+1)^{\frac{\alpha}{2(\alpha+1)}}\|P_{k}(g)\|_{L^{2}(\Omega, {\rm d}\Omega)}\\
\le&\,
 C(k+1)^{\frac{\alpha}{2(\alpha+1)}}\|g\|_{L^{2}(\Omega, {\rm d}\Omega)}
    \end{split}
\end{equation}
for every $g\in L^{2}(\Omega, {\rm d}\Omega)$.

On the other hand, the adjoint operator of $T_{\beta,k}$ is given by
\begin{eqnarray*}
T_{\beta, k}^{*}(g)(\phi, \omega_{1}, \omega_{2})=P_{k}\left(\psi^{-\beta}(\cdot)g\right)(\phi,\omega_{1}, \omega_{2}).
\end{eqnarray*}
Since $P_{k}$ is a projection operator, we can write
\begin{eqnarray*}
\psi^{-\beta}(\phi) P_{k}\left(\psi^{-\beta}(\cdot)g\right)(\phi,\omega_{1}, \omega_{2})&=&\psi^{-\beta}(\phi) P_{k}\circ P_{k}(\psi^{-\beta}(\cdot)g)(\phi,\omega_{1}, \omega_{2})\\
&=&T_{\beta, k}\circ T_{\beta,k}^{*}(g)(\phi,\omega_{1}, \omega_{2}).
\end{eqnarray*}
By \eqref{l2w2} and a standard duality argument for the norm of $T_{\beta,k}^{*}$, we obtain the desired   estimate \eqref{l2w1}.
\end{proof}
\begin{remark}
    Similar bounds to those in \eqref{l2w1} can be derived using Lemma \ref{nwf1} for the cases when  $n=2$, $m\ge 2$ even and $\alpha \in \mathbb{N} \backslash \{1\}$. Furthermore, by choosing an appropriate $\varepsilon$, it is still possible to obtain a bound $C(k+1)^{b}$ with $0<b<1$. That bound is sufficient to subsequently derive a Carleman-type estimate. 
\end{remark}
The $L^{p}-L^{q}$ estimates for the projector $P_{k}$ in \eqref{int1} follow from Theorem \ref{wl1} ($L^{1}-L^{\infty}$ estimates) and Theorem \ref{wl2} ($L^{2}-L^{2}$ estimates) via standard complex interpolation. We omit the proof.

\begin{theorem} \label{pq1} When {\rm (1)} $n\ge 4$, $m\ge 2$ and $\alpha\in \mathbb{N}$, or {\rm (2)} $n\in \{2, 3\}$, $m\ge 2$ and $\alpha=1$, let 
\begin{equation*}
    p=\frac{2\left(n+m-2+\frac{1}{\alpha+1}\right)}{n+m-2+\frac{2}{\alpha+1}}  \quad \mbox{and}\quad  q=\frac{2\left(n+m-2+\frac{1}{\alpha+1}\right)}{n+m-2}.
\end{equation*}
If $\displaystyle 0\le \beta<1/q$, there exists a constant $C>0$ only depending  on $\alpha, \beta, n$ and $m$, such that for $g\in L^{p}(\Omega, {\rm d}\Omega)$,
\begin{equation}\label{pq}
    \begin{split}
        &\left\|\psi^{-\beta}(\phi) P_{k}\left(\psi^{-\beta}(\cdot)g\right)\right\|_{L^{q}(\Omega, {\rm d}\Omega)}
        \le C(k+1)^{\frac{n+m-2}{n+m-2+\frac{1}{\alpha+1}}}\|g\|_{L^{p}(\Omega, {\rm d}\Omega)}.\\
    \end{split}
\end{equation}\\
\end{theorem}
\subsection{Tighter estimates on $\mathbb{R}^{n+1}$}  In this subsection, we present a  tighter weighted $L^{2}-L^{2}$  estimate for the the projector $P_{k}$ in \eqref{int1} over $\mathbb{R}^{n+1}$.  Our proof is based on  explicit formulas of weighted $L^{2}$-norms of Gegenbauer polynomials.  In particular,  we improve the estimate for the two step Grushin operator $\Delta_{1}$ obtained in \cite[Theorem 4.1]{GS}. The distinction between even and odd dimensions there seems not essential and is now removed.

The main result  of this subsection is as follows.
\begin{theorem} \label{wr1} Assume  $n\ge 2$.
There exists a constant $C>0$ depending only on $n$, $\alpha$ and $\beta$, such that, for every $g\in L^{2}(\Omega, {\rm d} \Omega)$, it holds 
\begin{equation} \label{erf1}
\begin{split}
     \int_{\Omega}\left|(\sin\phi)^{-\beta} P_{k}\left( \sin^{-\beta}(\cdot)g\right)(\phi, \omega)  \right|^{2}{\rm d}\Omega
     \le
    \left\{
  \begin{array}{ll} \displaystyle C\int_{\Omega}|g|^{2}\,{\rm d}\Omega, &0\le \beta<\frac{1}{2},\\
 \displaystyle C(k+1)^{2(2\beta-1)}\int_{\Omega}|g|^{2}\,{\rm d}\Omega, &\frac{1}{2}<\beta<1.
   \end{array}
\right. 
\end{split}  
\end{equation}
In particular, when $\alpha\in \mathbb{N} \backslash \{1\}$ and for $(\alpha+1)/4\alpha \le\beta<1/2$, there exists  a constant $C>0$, depending  on $\alpha, \beta$ and $ n$, such that 
\begin{equation}\label{fz1}
    \int_{\Omega}\left|\psi^{-\beta}(\phi) P_{k}\left(\psi^{-\beta}(\cdot)g\right)(\phi, \omega_{1}, \omega_{2})\right|^{2}\,{\rm d}\Omega\le C(k+1)^{\frac{2(\alpha-1)}{\alpha+1}}\int_{\Omega}|g|^{2}\,{\rm d}\Omega.
\end{equation}
\end{theorem}
\begin{remark}  Eq.\,\eqref{fz1} is a direct corollary of the second case in Eq.\,\eqref{erf1}.
\end{remark}
\begin{remark} When $\alpha=1$, the first case in \eqref{erf1} recovers  the important estimate in \cite[Theorem 4.1]{GS} for even $n$. The estimate for odd $n$ there is now improved.
\end{remark}
Recall that for any  $k\in \mathbb{N}_{0}$, we have defined
\begin{equation}
    h_{k, \ell}(\phi)=b_{k,\ell}\sin^{\frac{\ell}{\alpha+1}}\phi\, C_{\frac{k-\ell}{\alpha+1}}^{\frac{2\ell+n-2}{2(\alpha+1)}+\frac{1}{2}}(\cos \phi), \quad 0\le \ell\le k,\;  \ell\equiv k\,({\rm mod}(\alpha+1)),   
\end{equation}
where $b_{k,\ell}$ is the normalization constant such that $\int_{0}^{\pi/2} \left(h_{k, \ell}(\phi)\right)^{2} \sin^{\frac{n-2}{\alpha+1}+1}\phi\, {\rm d}\phi=1$. Then the set 
\begin{equation*}
    \left\{h_{k,\ell}(\phi)Y_{\ell, j}(\omega)\,\arrowvert\, 0\le \ell\le k, \ell\equiv k\,({\rm mod}(\alpha+1)), 1\le j\le d_{\ell}\right\}
\end{equation*}
is an orthonormal basis for $\mathcal{H}_{k}^{\alpha}$ over $\partial B_{1}$, see Remark \ref{res1}. Similarly as done in Section \ref{421} as well as in \cite[\S 4]{GS}, it is seen that Theorem \ref{wr1} will follow if we can show,
\begin{lemma}\label{pqq} Assume  $n\ge 2$. There exists $C>0$ only depending on $\alpha, \beta$ and $n$, such that
    \begin{equation}\label{wl22}
    \int_{0}^{\pi}\left|h_{k, \ell}(\phi)\right|^{2} (\sin \phi)^{\frac{n-2}{\alpha+1}+1-2\beta}\, {\rm d}\phi
      \le
    \left\{
  \begin{array}{ll} \displaystyle C, &0\le \beta<\frac{1}{2},\\
 \displaystyle C(k+1)^{2\beta-1}, &\frac{1}{2}<\beta<1.
   \end{array}
\right. 
\end{equation}
\end{lemma} 

To establish \eqref{wl22}, we derive some estimates valid for  Gegenbauer polynomials. First, we
recall the following results in \cite[Theorem 5 and 6]{bg}.
\begin{theorem} \label{asym1}  Let $\lambda>0$ and $\mu >-1/2$. Then we have
\begin{equation} \label{gexc}
\begin{split}
     J_{j}^{(\lambda; \mu)}:= &\int_{0}^{\pi}\left|C_{j}^{(\lambda)}(\cos \theta)\right|^{2}(\sin \theta)^{2\mu}\,{\rm d}\theta\\
    =&\frac{\sqrt{\pi}\Gamma\left(\mu+\frac{1}{2}\right)}{\mu\Gamma(\mu+1)}\sum_{k=0}^{\lfloor\frac{j}{2}\rfloor}
    \frac{(\lambda)_{j-k}^{2}(\lambda-\mu)_{k}^{2}}{(\mu+1)_{j-k}^{2}(k!)^{2}}(j+\mu-2k)\frac{(2\mu)_{j-2k}}{(j-2k)!}.
\end{split}
\end{equation}
Furthermore, $J_{j}^{(\lambda; \mu)}$ satisfies the asymptotic estimate 
\begin{equation} \label{gexcc1}
    J_{j}^{(\lambda; \mu)}=
    \left\{
  \begin{array}{ll}
  \displaystyle  \frac{\sqrt{\pi}\Gamma(\mu+\frac{1}{2}-\lambda)}{2^{2\lambda-1}\Gamma(\lambda)^{2}\Gamma(\mu+1-\lambda)}
    j^{2\lambda-2}+\mathcal{O}(j^{\eta}), & \mu>\lambda-\frac{1}{2}, \\
 \displaystyle  \frac{\sqrt{\pi}\Gamma(\lambda-\mu-\frac{1}{2})\Gamma(\mu+\frac{1}{2})}{2^{2\lambda-1}\Gamma(\lambda)^{2}\Gamma(\lambda-\mu)\Gamma(2\lambda-\mu-\frac{1}{2})}
    j^{4\lambda-2\mu-3}+\mathcal{O}(j^{\eta}), & \mu<\lambda-\frac{1}{2}, 
  \end{array}
\right.   
\end{equation}
with \begin{equation*} \eta= \left\{
  \begin{array}{ll}\max(2\lambda-3, 4\lambda-2\mu-3), &\mu>\lambda-\frac{1}{2},\\
  \max(2\lambda-2, 4\lambda-2\mu-4), &\mu<\lambda-\frac{1}{2}.
   \end{array}
\right. 
  \end{equation*}
\end{theorem}
\begin{remark}
    The explicit weighted $L^{2}$-norm  of Gegenbauer polynomials in \eqref{gexc} is again derived  by  the connection formula (see \cite[(4.10.27)]{or})
   \begin{equation*}
       C_{j}^{(\lambda)}(x)=\sum_{k=0}^{\lfloor \frac{j}{2}\rfloor}\frac{(\lambda)_{j-k}(\lambda-\mu)_{k}}{(\mu+1)_{j-k}k!}
       \frac{j+\mu-2k}{\mu}C_{j-2k}^{(\mu)}(x),
   \end{equation*}
    and the orthogonality relation \eqref{or1}.
\end{remark}
     For convenience, we set 
\begin{equation*}
\begin{split}
     b^{(\lambda)}_{j}:=\left(\int_{0}^{\pi}\left(C_{j}^{(\lambda)}(\cos \theta)\right)^{2}\sin^{2\lambda}\theta {\rm d}\theta \right)^{-\frac{1}{2}}
     =\left(\frac{2^{1-2\lambda}\pi\Gamma(j+2\lambda)}{[\Gamma(\lambda)]^{2}(j+\lambda)\Gamma(j+1)} \right)^{-\frac{1}{2}}
\end{split} 
\end{equation*}
for $\lambda>-1/2$ and $\lambda\neq 0$. 

\begin{theorem} \label{iin}  Let $\lambda>0$ and $\mu >-1/2$.  There exists a constant $C>0$ only depending on $\lambda$ and $\mu$, such that
 \begin{equation*}
 \begin{split}
        &\int_{0}^{\pi} \left|b_{j-\ell}^{(\lambda+\ell)}C_{j-\ell}^{(\lambda+\ell)}(\cos \theta)\right|^{2}(\sin \theta)^{2(\mu+\ell)}\,{\rm d}\theta\\
        \le& 
     \left\{
  \begin{array}{ll}  C, &\displaystyle  \lambda-1/2< \mu\le \lambda,\\
 C(j+1)^{2(\lambda-\mu)-1}, & \displaystyle\mu<\lambda-1/2.
   \end{array}
\right. 
 \end{split}
 \end{equation*}
 for all $j,\ell \in \mathbb{N}_{0}$ and $j\ge \ell$.
\end{theorem}
\begin{proof} We prove the claim in two steps. First we prove the results for $\ell=0$ and all $j\in \mathbb{N}_{0}$. Then we derive the remaining by a decreasing relation. 

{\rm (\uppercase\expandafter{\romannumeral1})} When $ \lambda-1/2< \mu\le \lambda$ and $\ell=0$, by Theorem \ref{asym1}, there exists a constant $C(\lambda, \mu)$ such that
    \begin{equation} \label{rat3}
    \begin{split}
       \int_{0}^{\pi} \left|b_{j}^{(\lambda)}C_{j}^{(\lambda)}(\cos \theta)\right|^{2}(\sin \theta)^{2\mu}\,{\rm d}\theta
     =&\,\left(\frac{2^{1-2\lambda}\pi\Gamma(j+2\lambda)}{[\Gamma(\lambda)]^{2}(j+\lambda)\Gamma(j+1)} \right)^{-1}\\
     &\times\left(\frac{\sqrt{\pi}\Gamma(\mu+\frac{1}{2}-\lambda)}{2^{2\lambda-1}\Gamma(\lambda)^{2}\Gamma(\mu+1-\lambda)}
    j^{2\lambda-2}+\mathcal{O}(j^{\eta})\right)\\  
     \le &\, C(\lambda, \mu)
    \end{split}
 \end{equation}
holds for all $j\in \mathbb{N}_{0}$. This can also be seen from  Askey's transplantation theorem \cite{as}.
Similarly, we obtain the estimate for $ \mu<\lambda-1/2$ for $\ell=0$.\\

{\rm (\uppercase\expandafter{\romannumeral2})} For $\mu<\lambda-1/2$ or $ \lambda-1/2< \mu\le \lambda$,  we consider the following ratio
\begin{equation}\label{rat1}
\begin{split}
     &\frac{\left(b_{j-1}^{(\lambda+1)}\right)^{2}J_{j-1}^{\left(\lambda+1; \mu+1\right)}}{\left(b_{j}^{(\lambda)}\right)^{2} J_{j}^{(\lambda; \mu)}}=\frac{\int_{0}^{\pi} \left|b_{j-1}^{(\lambda+1)}C_{j-1}^{(\lambda+1)}(\cos \theta)\right|^{2}(\sin \theta)^{2(\mu+1)}\,{\rm d}\theta}{\int_{0}^{\pi} \left|b_{j}^{(\lambda)}C_{j}^{(\lambda)}(\cos \theta)\right|^{2}(\sin \theta)^{2\mu}\,{\rm d}\theta}\\
    =&\frac{4\lambda^{2}}{j(j+2\lambda)}\cdot\frac{\mu\left(\mu+\frac{1}{2}\right)}{(\mu+1)^{2}}
    \cdot \frac{\sum_{k=0}^{\lfloor\frac{j-1}{2}\rfloor}
    \frac{(\lambda+1)_{j-1-k}^{2}(\lambda-\mu)_{k}^{2}}{(\mu+2)_{j-1-k}^{2}(k!)^{2}}(j+\mu-2k)\frac{(2\mu+2)_{j-1-2k}}{(j-1-2k)!}}{\sum_{k=0}^{\lfloor\frac{j}{2}\rfloor}
    \frac{(\lambda)_{j-k}^{2}(\lambda-\mu)_{k}^{2}}{(\mu+1)_{j-k}^{2}(k!)^{2}}(j+\mu-2k)\frac{(2\mu)_{j-2k}}{(j-2k)!}}.
\end{split}
\end{equation}
with $j\ge 1$. 

Note that each term in the right-hand side of \eqref{gexc} is positive (due to $\mu>-1/2$ and $(\mu)_{0}=1$). For a fixed $k$ satisfying  $0\le k\le \lfloor(j-1)/2)\rfloor$,  we check the ratios of the corresponding terms in  the two summations of the right-hand side of \eqref{rat1},
\begin{equation}\label{rat2}
    \begin{split}
        &\frac{4\lambda^{2}}{j(j+2\lambda)}\cdot\frac{\mu\left(\mu+\frac{1}{2}\right)}{(\mu+1)^{2}}
    \cdot \frac{
    \frac{(\lambda+1)_{j-1-k}^{2}(\lambda-\mu)_{k}^{2}}{(\mu+2)_{j-1-k}^{2}(k!)^{2}}(j+\mu-2k)\frac{(2\mu+2)_{j-1-2k}}{(j-1-2k)!}}{
    \frac{(\lambda)_{j-k}^{2}(\lambda-\mu)_{k}^{2}}{(\mu+1)_{j-k}^{2}(k!)^{2}}(j+\mu-2k)\frac{(2\mu)_{j-2k}}{(j-2k)!}}\\
    =& \,\frac{4\lambda^{2}}{j(j+2\lambda)}\cdot\frac{\mu\left(\mu+\frac{1}{2}\right)}{(\mu+1)^{2}}\cdot
     \frac{(\mu+1)^{2}(2\mu+j-2k)(j-2k)}{2\mu(2\mu+1)\lambda^{2} }\\
     =&\,\frac{(j-2k)(j+2\mu-2k)}{j(j+2\lambda)}\\
     \le&\, 1.
    \end{split}
\end{equation}
By \eqref{rat2} and the fact $\lfloor(j-1)/2\rfloor\le \lfloor j/2 \rfloor$ as well as \eqref{rat3}, we obtain 
\begin{equation*}
\begin{split}
     &\int_{0}^{\pi} \left|b_{j-1}^{(\lambda+1)}C_{j-1}^{(\lambda+1)}(\cos \theta)\right|^{2}(\sin \theta)^{2(\mu+1)}\,{\rm d}\theta\\
     \le&\, \int_{0}^{\pi} \left|b_{j}^{(\lambda)}C_{j}^{(\lambda)}(\cos \theta)\right|^{2}(\sin \theta)^{2\mu}\,{\rm d}\theta\\
     \le&\, C(\lambda, \mu), 
\end{split}
\end{equation*}
for  all $j-1\in \mathbb{N}_{0}$. Similarly, we obtain the estimate for $\mu<\lambda-1/2$.
By induction, we complete the proof.
\end{proof}
\begin{remark} Theorem \ref{iin} is of  independent interest. The first estimate can be considered as a refined version of Askey's transplantation theorem in $L^{2}$ setting \cite{as}.
\end{remark}
\begin{remark}
  We see  from \eqref{gexcc1} that  the growth order in the second case is optimal.
\end{remark}

Now we go back to the weighted $L^{2}-L^{2}$ estimate in the Grushin setting.
\begin{proof}[Proof of Lemma \ref{pqq}] We consider the case $0\le \beta<1/2$ first. 
For each $\ell\equiv 0 \,({\rm mod}(\alpha+1))$,   the weighted norm of
   \begin{eqnarray*}
      h_{k, \ell}(\phi)&=& b_{k,\ell}\sin^{\frac{\ell}{\alpha+1}}\phi\, C_{\frac{k-\ell}{\alpha+1}}^{\frac{2\ell+n-2}{2(\alpha+1)}+\frac{1}{2}}(\cos \phi)\\&=& b_{k,\ell}\sin^{\frac{\ell}{\alpha+1}}\phi\, C_{\frac{k-\ell}{\alpha+1}}^{\frac{\ell}{\alpha+1}+\left(\frac{n-2}{2(\alpha+1)}+\frac{1}{2}\right)}(\cos \phi). 
 \end{eqnarray*}
can be estimated by  setting $\lambda=\frac{n-2}{2(\alpha+1)}+\frac{1}{2}$ and $\mu=\lambda-\beta$ in Theorem \ref{iin}. Thus  there exists a constant $C_{0}>0$ such that 
\begin{equation*}
    \int_{0}^{\pi}\left|h_{k, \ell}(\phi)\right|^{2} (\sin \phi)^{\frac{n-2}{\alpha+1}+1-2\beta}\, {\rm d}\phi\le C_{0},  
\end{equation*}
where both $\ell$ and $k$ are congruent to $0$ modulo $\alpha+1$   and  $0\le \ell\le k$.
Similarly, we can find constants $C_{j}$ for other $\ell\equiv j({\rm mod}(\alpha+1))$ with $j=1,2,\ldots,\alpha$. 
Let $C=\max_{0\le j\le \alpha}\{C_{j}\}$, we find the first constant for \eqref{wl22}.

The second part of \eqref{wl22} can be proven in the same way. This concludes the proof.
\end{proof}

\section{Carleman estimates and unique continuation}
\subsection{Carleman inequalities}\label{sec5}
Using the weighted $L^{p}-L^{q}$ bounds from Theorem \ref{pq1} for the projector $P_{k}$ in \eqref{int1},  we establish  Carleman estimates for the higher step Grushin operator $\Delta_{\alpha}$ in this subsection. 

Recall that $\psi(\phi)=\sin^{\frac{2\alpha}{\alpha+1}}\phi$ and $Q=n+m(\alpha+1)$. 
We first consider the cases when (1) $n\ge 4$, $m\ge 2$ and $\alpha\in \mathbb{N}$, or (2) $n\in\{2, 3\}$, $m\ge 2$ and $\alpha=1$. Let
\begin{equation} \label{pqd1}
    p=\frac{2\left(n+m-2+\frac{1}{\alpha+1}\right)}{n+m-2+\frac{2}{\alpha+1}} \quad \mbox{and}\quad  q=\frac{2\left(n+m-2+\frac{1}{\alpha+1}\right)}{n+m-2}.
\end{equation}
It is seen that
\begin{equation*}
    \frac{1}{p}+\frac{1}{q}=1 \quad \mbox{and} \quad \frac{1}{p}-\frac{1}{q}=\frac{1}{(\alpha+1)(n+m-2)+1}.
\end{equation*}
The main result of this subsection is as follows.
\begin{theorem} \label{sxc1}
Let $0<\varepsilon<1/8$, $s>100$, $\delta={\rm dist}\,(s, \mathbb{N})>0$, and consider $p$ and $q$ as defined in \eqref{pqd1}.  Then there exists a constant $C>0$ depending only on $\alpha, \varepsilon, \delta, n$ and $m$, such that
\begin{equation}\label{z2}
    \begin{split}
       &\left\|\rho^{-s}\psi^{\varepsilon} g\right\|_{L^{q}\left(\mathbb{R}^{n+m}, \rho^{-Q}\,{\rm d}x\,{\rm d}y\right)}
       \le C\left\|\rho^{-s+2}\psi^{-\varepsilon}\Delta_{\alpha}(g) \right\|_{L^{p}\left(\mathbb{R}^{n+m}, \rho^{-Q}\,{\rm d}x\,{\rm d}y\right)} \\
    \end{split}
\end{equation} for $g\in C_{0}^{\infty}(\mathbb{R}^{n+m}\backslash \{0\})$. 
\end{theorem}

\begin{proof} We  prove it in four steps. Our proof follows the approach  developed in \cite{JD} by Jerison, which was  later extended to the Grushin setting by Garofalo and Shen in  \cite[Theorem 5.1]{GS}.

{\rm (\uppercase\expandafter{\romannumeral1})} We  express the desired estimate \eqref{z2} in an equivalent form.
Let $\rho=e^{t}, t\in \mathbb{R}$. The Grushin operator $\Delta_{\alpha}$ in \eqref{ps} now can be written as
\begin{eqnarray}\label{z1}
\Delta_{\alpha}=e^{-2t}\psi(\phi) \cdot\left(\frac{\partial^{2}}{\partial t^{2}}+(Q-2)\frac{\partial}{\partial t}
+\Delta_{\sigma} \right),\end{eqnarray}
where $\Delta_{\sigma}$ is given in \eqref{r1}. Thus for $u_{k}(\phi, \omega)\in \mathcal{H}_{k}^{\alpha}$  where $\omega\in \mathbb{S}^{n-1}\times \mathbb{S}^{m-1}$, by \eqref{f3} and \eqref{z1}, we have
\begin{eqnarray}\label{egg1}
\rho^{-s+2}\Delta_{\alpha}(\rho^{s}e^{it\eta}u_{k})&=&e^{(-s+2)t}\Delta_{\alpha}(e^{st}e^{it\eta}u_{k})\\
&=&-(k-(s+i\eta))(k+Q-2+s+i\eta)\,\psi \,e^{it\eta}\,u_{k}.\nonumber
\end{eqnarray}
For later use, we denote \[a_{s}(\eta, k)=-\left((k-(s+i\eta))(k+Q-2+s+i\eta)\right)^{-1}.\] Moreover, for any $f(t, \phi, \omega)=g(e^{t}, \phi, \omega)\in C_{0}^{\infty}(\mathbb{R}\times \Omega)$, denote
\[\hat{f}(\eta, \phi, \omega)=\int_{-\infty}^{\infty}e^{-it\eta}f(t, \phi, \omega)\,{\rm d}t.\]
 Similarly as done in \cite[(5.11)]{GS} and \cite[p.\,130]{JD},  
we introduce the following operator,
\begin{equation*}
    \begin{split}
       R_{s}(f)(t, \phi, \omega):=&\sum_{k=0}^{\infty}\frac{1}{2\pi}\int_{-\infty}^{\infty}a_{s}(\eta, k)e^{it\eta}
Q_{k}(\hat{f})(\eta, \phi,\omega)\,{\rm d}\eta\\
=&\sum_{k=0}^{\infty}\frac{1}{2\pi}\int_{\mathbb{R}}\left(\int_{\mathbb{R}}a_{s}(\eta, k)e^{i(t-w)\eta}Q_{k}\,{\rm d}\eta \right)f(w, \phi, \omega)\,{\rm d} w,   
    \end{split}
\end{equation*}
where $Q_{k}$ is the operator given by
\begin{equation*}
    Q_{k}(f)(\eta, \phi, \omega):=P_{k}\left(\frac{f(\eta, \cdot, \cdot)}{\psi(\cdot)}\right)(\phi,\omega).
\end{equation*}
The operator $R_{s}$ is the inverse of  the operator  $L_{s}(f):=\rho^{-s+2}\Delta_{\alpha}(\rho^{s}f)$ at least formally. 
It can be formally verified by \eqref{egg1} and the  fact $f(t)=(1/2\pi)\int_{-\infty}^{\infty}e^{its}\hat{f}(s)\,{\rm d} s$.

Moreover,  it is not hard to see from  Eq.\,\eqref{dm1} for  ${\rm d}\Omega$   and Eq.\,\eqref{mea} for ${\rm d}x \,{\rm d}y$ that 
\begin{equation*}
     \rho^{-Q}\,{\rm d}x\,{\rm d}y=\frac{1}{(\alpha+1)^{m}}\psi^{-1}\,{\rm d}t\,{\rm d}\Omega.
\end{equation*}
Collecting  everything,  we see that the estimate \eqref{z2} is equivalent to 
\begin{equation} \label{wp1}
    \begin{split}
      &\|R_{s}(f)\|_{L^{q}\left(\mathbb{R}\times\Omega, \psi^{-1+\varepsilon q}\,{\rm d}t\,{\rm d}\Omega\right)} \le C\|f\|_{L^{p}\left(\mathbb{R}\times\Omega, \psi^{-1-\varepsilon p}\,{\rm d}t\,{\rm d}\Omega\right)}\\
    \end{split}
\end{equation}
for all $f\in C_{0}^{\infty}(\mathbb{R}\times\Omega)$.

{\rm (\uppercase\expandafter{\romannumeral2})}  In this step,  we introduce new operators based on a partition of unity in order to prove \eqref{wp1}. Fix $s>100$ with ${\rm dist}(s, \mathbb{N})=\delta>0$. Assume $2^{N}\le (s/10)<2^{N+1}$. 
We choose the same partition of unity
$\{\Phi_{\gamma}\}_{\gamma=0}^{N}$ for $\mathbb{R}_{+}$ as in \cite[(5.14)]{GS} and \cite[(9)]{JD}, such that
\begin{eqnarray}
\left\{
  \begin{array}{ll}
    \sum_{\gamma=0}^{N}\Phi_{\gamma}(r)=1, & \hbox{all $r>0$;} \\
    {\rm supp}\,\Phi_{\gamma}\subset\{r: 2^{\gamma-2}\le r\le 2^{\gamma}\}, & \hbox{$\gamma=1,2,\ldots, N-1$;} \\
    {\rm supp}\,\Phi_{0}\subset\{r: 0<r\le 1\}, & \\
    {\rm supp}\,\Phi_{N}\subset\{r: r\ge s/40\}, &
  \end{array}
\right.
\end{eqnarray}
and satisfying
\begin{eqnarray}\label{pdx1}
\left|\frac{{\rm d}^{\ell}}{{\rm d}r^{\ell}}\Phi_{\gamma}(r)\right|\le C_{\ell}2^{-\gamma\ell}, \quad \ell=0,1,2,\ldots.
\end{eqnarray}
For $0\le \gamma\le N$,  we let
\begin{eqnarray} \label{pdx2}
a_{s}^{\gamma}(\eta, k):=\Phi_{\gamma}(|k-s-i\eta|)a_{s}(\eta, k)
\end{eqnarray}
and define
\begin{eqnarray*}
R_{s}^{\gamma}(f)(t, \phi, \omega)&:=&\sum_{k=0}^{\infty}\frac{1}{2\pi}\int_{-\infty}^{\infty}a_{s}^{\gamma}(\eta, k)e^{it\eta}
 Q_{k}(\tilde{f})(\eta, \phi, \omega)\,{\rm d}\eta\\
 &=&\sum_{k=0}^{\infty}\frac{1}{2\pi}\int_{\mathbb{R}}\left(\int_{\mathbb{R}}a_{s}^{\gamma}(\eta, k)e^{i(t-w)\eta}Q_{k}\,{\rm d}\eta \right)f(w, \phi, \omega)\,{\rm d} w.
\end{eqnarray*}
Obviously, it holds that
\begin{equation*}
    R_{s}=\sum_{\gamma=0}^{N}R_{s}^{\gamma}.
\end{equation*}

{\rm (\uppercase\expandafter{\romannumeral3})} We estimate $R_{s}^{\gamma}$ for $0\le \gamma \le N-1$ in this step. 
 Observing when $0\le \gamma\le N-1$,  $a_{s}^{\gamma}(\eta, k)$ is supported where $\delta 2^{\gamma-2}\le|k-s-i\eta|\le 2^{\gamma}$. Therefore,  there are at most $2\cdot 2^{\gamma}$ non-zero terms in the sum over $k$ which defines $R_{s}^{\gamma}$. Furthermore, the value of $k$ in each case is  comparable to $s$. Now, by Eqs.\,\eqref{pdx1}, \eqref{pdx2} and the aforementioned facts, it can be verified  that 
\begin{eqnarray} \label{coa}
\left|\left(\frac{\partial}{\partial \eta}\right)^{j}a_{s}^{\gamma}(\eta, k)\right|\le C_{j}\cdot2^{-(1+j)\gamma}\cdot s^{-1},\qquad j=0,1,\ldots.
\end{eqnarray}

The $L^{p}-L^{q}$ estimate \eqref{pq}   in Theorem \ref{pq1} implies that
\begin{eqnarray*}
\left\| Q_{k}(f)\right\|_{L^{q}\left(\Omega, \psi^{-\beta q}{\rm d}\Omega\right)}\le C(k+1)^{\frac{n+m-2}{n+m-2+\frac{1}{\alpha+1}}}\left\|f \right\|_{L^{p}\left(\Omega, \psi^{\beta p-p}{\rm d}\Omega\right)}
\end{eqnarray*}
for $0\le \beta<1/q$. Letting $\varepsilon=1/q-\beta$, we get
\begin{equation}\label{pqd}
    \begin{split}
        &\left\|Q_{k}(f)\right\|_{L^{q}\left(\Omega, \psi^{-1+\varepsilon q}\,{\rm d}\Omega\right)}\\ \le&\,
        C(k+1)^{\frac{n+m-2}{n+m-2+\frac{1}{\alpha+1}}}\left\|f \right\|_{L^{p}\left(\Omega, \psi^{-1-\varepsilon p}\,{\rm d}\Omega\right)}.
    \end{split}
\end{equation}
Now performing integration by parts, we conclude from \eqref{coa} and \eqref{pqd} that
\begin{eqnarray*}
&&\left\|\int_{\mathbb{R}}e^{i(t-w)\eta}\sum_{k=0}^{\infty}a_{s}^{\gamma}(\eta, k)Q_{k}(f)\,{\rm d}\eta\right\|_{L^{q}\left(\Omega, \psi^{-1+\varepsilon q}\,{\rm d}\Omega\right)}\\
&\le&\frac{C}{|t-w|^{j}}\sum_{k=0}^{\infty}\int_{\mathbb{R}}\left|\left(\frac{\partial}{\partial \eta}\right)^{j}a_{s}^{\gamma}(\eta,k)\right|\,{\rm d}\eta \|Q_{k}(f)\|_{L^{q}\left(\Omega, \psi^{-1+\varepsilon q}\,{\rm d}\Omega\right)}\\
&\le&\frac{C}{|t-w|^{j}}\cdot 2^{-j\gamma}\cdot 2^{\gamma}\cdot s^{-1+\frac{n+m-2}{n+m-2+\frac{1}{\alpha+1}}}\|f\|_{L^{p}\left(\Omega, \psi^{-1-\varepsilon p}\,{\rm d}\Omega\right)}\\
&=&\frac{C}{\left(2^{\gamma}|t-w|\right)^{j}}\cdot 2^{\gamma}\cdot s^{-\frac{1}{(\alpha+1)(n+m-2)+1}}\|f\|_{L^{p}\left(\Omega, \psi^{-1-\varepsilon p}\,{\rm d}\Omega\right)}.
\end{eqnarray*}
Here the second inequality is because the integration in $\eta$ is over an interval of length $\le 2\cdot 2^{\gamma}$ and there are at most $2\cdot 2^{\gamma}$ non-zero terms in the sum over $k$. 

Now choosing $j=0$ and $j=10$, we obtain
\begin{eqnarray*}
&&\left\|\int_{\mathbb{R}}e^{i(t-w)\eta}\sum_{k=0}^{\infty}a_{s}^{\gamma}(\eta, k)Q_{k}(f)\,{\rm d}\eta\right\|_{L^{q}\left(\Omega, \psi^{-1+\varepsilon q}\,{\rm d}\Omega\right)}\\
&\le&\frac{C}{(1+2^{\gamma}|t-w|)^{10}}\cdot s^{-\frac{1}{(\alpha+1)(n+m-2)+1}}\cdot 2^{\gamma}\|f\|_{L^{p}\left(\Omega, \psi^{-1-\varepsilon p}\,{\rm d}\Omega\right)}.
\end{eqnarray*}
Thus, for $0\le\gamma\le N-1$, we have,
\begin{eqnarray*}
&&\left\|R_{s}^{\gamma}(f)\right\|_{L^{q}\left(\mathbb{R}\times\Omega, \psi^{-1+\varepsilon q}{\rm d}t{\rm d}\Omega\right)}\\
&\le& C\cdot s^{-\frac{1}{(\alpha+1)(n+m-2)+1}}\cdot 2^{\gamma}\left\|\int_{\mathbb{R}}\frac{1}{(1+2^{\gamma}|t-w|)^{10}}\|f(t,\cdot, \cdot)\|_{L^{p}\left(\Omega, \psi^{-1-\varepsilon p}{\rm d}\Omega\right)}{\rm d}t\right\|_{L^{q}(\mathbb{R}, {\rm d}t)}\\
&\le& C\cdot s^{-\frac{1}{(\alpha+1)(n+m-2)+1}}\cdot 2^{\gamma}
\left\|\frac{1}{(1+2^{\gamma}|\cdot|)^{10}}\right\|_{L^{\frac{n+m-2+1/(\alpha+1)}{n+m-2}}(\mathbb{R},{\rm d}t)}
\|f\|_{L^{p}\left(\mathbb{R}\times\Omega,\psi^{-1-\varepsilon p}{\rm d}t{\rm d}\Omega\right)}\\
&\le& C\cdot s^{-\frac{1}{(\alpha+1)(n+m-2)+1}}\cdot 2^{\frac{\gamma}{(\alpha+1)(n+m-2)+1}}\|f\|_{L^{p}\left(\mathbb{R}\times\Omega,\psi^{-1-\varepsilon p}{\rm d}t{\rm d}\Omega\right)},
\end{eqnarray*}
where the first inequality is by  Minkowski's inequality and the second inequality is by Young's convolution inequality, respectively.
Since
\begin{equation*}
    \sum_{\gamma=0}^{N-1}2^{\frac{\gamma}{(\alpha+1)(n+m-2)+1}}\le s^{(\alpha+1)(n+m-2)+1)^{-1}},
\end{equation*}
 we have
\begin{eqnarray*}
\sum_{\gamma=0}^{N-1}\left\|R_{s}^{\gamma}(f)\right\|_{L^{q}\left(\mathbb{R}\times\Omega, \psi^{-1+\varepsilon q}\,{\rm d}t\,{\rm d}\Omega\right)}
\le C \|f\|_{L^{p}\left(\mathbb{R}\times\Omega,\psi^{-1-\varepsilon p}\,{\rm d}t\,{\rm d}\Omega\right)}.
\end{eqnarray*}

{\rm (\uppercase\expandafter{\romannumeral4})} We still need to estimate $R_{s}^{N}(g)$. 
It can be checked that
\begin{eqnarray}\label{ls1}
\left|\left(\frac{\partial}{\partial\eta}\right)^{j}a_{s}^{N}(\eta,k)\right|\le \frac{C_{j}}{(|\eta|+s+k)^{j+2}}.
\end{eqnarray}
By \eqref{ls1} and integration by parts, choosing $j=0$ and $j=1$, it follows that
\begin{eqnarray*}
\left|\int_{\mathbb{R}}e^{i(t-w)\eta}a_{s}^{N}(\eta,k)\,{\rm d}\eta \right|\le \frac{C}{(k+s)[1+|t-w|(k+s)]}.
\end{eqnarray*}
Furthermore, by the above estimate and the bound in \eqref{pqd} for $Q_{k}$, we have
\begin{eqnarray*}
&&\left\|\sum_{k=0}^{\infty}\int_{\mathbb{R}}e^{i(t-w)\eta}a_{s}^{N}(\eta,k)Q_{k}(f)\,{\rm d}\eta \right\|_{L^{q}\left(\Omega,\psi^{-1+\varepsilon q}{\rm d}t{\rm d}\Omega\right)}\\
&\le& C\sum_{k=0}^{\infty}\frac{(k+1)^{\frac{n+m-2}{n+m-2+\frac{1}{\alpha+1}}}}{(k+s)[1+|t-w|(k+s)]}
\|f\|_{L^{p}\left(\Omega,\psi^{-1-\varepsilon p}{\rm d}t{\rm d}\Omega\right)}\\
&\le&C
\left\{ \sum_{k\le\frac{1}{|t-w|}}(k+1)^{-\frac{1}{(\alpha+1)(n+m-2)+1}}+\sum_{k>\frac{1}{|t-w|}}(k+1)^{-1-\frac{1}{(\alpha+1)(n+m-2)+1}}|t-w|^{-1}\right\}\\
&&\times \|f\|_{L^{p}\left(\Omega,\psi^{-1-\varepsilon p}{\rm d}t{\rm d}\Omega\right)}\\
&\le&\frac{C}{|t-w|^{\frac{n+m-2}{n+m-2+\frac{1}{\alpha+1}}}}\|f\|_{L^{p}\left(\Omega,\psi^{-1-\varepsilon p}{\rm d}t{\rm d}\Omega\right).}
\end{eqnarray*}
Now, the desired estimate  for the last term $R_{s}^{N}$  follows by applying the Minkowski's inequality and using the well-known bounds on fractional integration.

Combining the estimates in {\rm (\uppercase\expandafter{\romannumeral3})} and {\rm (\uppercase\expandafter{\romannumeral4})}, we conclude the proof.
\end{proof}

Using the  weighted $L^{2}-L^{2}$ estimate obtained in \eqref{fz1} and the $L^{1}-L^{\infty}$ estimate in Remark \ref{s1et}, similarly, we have the  Carleman estimates for
$\Delta_{\alpha}$ on $\mathbb{R}^{n+1}$.
\begin{theorem} \label{sxc2}
Let $0<\varepsilon<1/8$, $s>100$ and $\delta={\rm dist}(s, \mathbb{N})>0$. For given $\alpha\in \mathbb{N}\backslash \{1\}$,  there exists a constant $C>0$ depending only on $\alpha, \varepsilon, \delta$ and $ n$, such that for $g\in C_{0}^{\infty}(\mathbb{R}^{n+1}\backslash \{0\})$, 
\begin{equation}
       \left\|\rho^{-s}\psi^{\varepsilon} g\right\|_{L^{q}\left(\mathbb{R}^{n+1}, \rho^{-Q}\,{\rm d}x\,{\rm d}y\right)}\\
\le
C\left\|\rho^{-s+2}\psi^{-\varepsilon}\Delta_{\alpha}(g) \right\|_{L^{p}\left(\mathbb{R}^{n+1}, \rho^{-Q}\,{\rm d}x\,{\rm d}y\right)}, 
\end{equation}
where  
\begin{equation*}
\left\{
  \begin{array}{ll}
    p=\frac{6\alpha+10}{3\alpha+7}, \quad q=\frac{6\alpha+10}{3\alpha+3}, & \hbox{if $n\in \{2, 3\}$;} \\
    p=\frac{4+2(\alpha+1)(n-1)}{4+(\alpha+1)(n-1)}, \quad q=\frac{4+2(\alpha+1)(n-1)}{(\alpha+1)(n-1)}, & \hbox{if $n\ge 4$.}
  \end{array}
\right.
\end{equation*}
\end{theorem}
As a summary, we collect the exponents $p$ and $q$ in the Carleman estimates for $\Delta_{\alpha}$ in Table \ref{tab:example1}.
\begin{table}
        \centering
         \caption{Exponents $p$ and $q$ in Carleman estimates}
         \label{tab:example1}
\begin{tabular}{cclcc}
   \toprule
      $m$ & $n$ &$\alpha$ & $p$  & $q$  \\
   \midrule
   \multirow{3}{*}{$m=1$} 
   & $n\ge 2$ &    1  & $\frac{2n}{n+1}$ &$\frac{2n}{n-1}$ \\  
   & $n\in\{2,3\}$  & $\mathbb{N}_{\ge 2}$  &  $\frac{6\alpha+10}{3\alpha+7}$ & $\frac{6\alpha+10}{3\alpha+3}$\\
   &  $n\ge 4$&  $\mathbb{N}_{\ge 2}$ &  $\frac{4+2(\alpha+1)(n-1)}{4+(\alpha+1)(n-1)}$  &$\frac{4+2(\alpha+1)(n-1)}{(\alpha+1)(n-1)}$ \\
     \hline
     \multirow{3}{*}{$m\ge 2$}  &$n=3$ & $\mathbb{N}_{\ge 2}$& $\frac{2\left(m+2+\frac{1}{\alpha+1}\right)}{m+2+\frac{2}{\alpha+1}}$ &$\frac{2\left(m+2+\frac{1}{\alpha+1}\right)}{m+2}$ \\
     \cline{2-5}
     & $n\in \{2,3\}$   & $1$ & \multirow{2}{*}{$\frac{2\left(n+m-2+\frac{1}{\alpha+1}\right)}{n+m-2+\frac{2}{\alpha+1}}$}  & \multirow{2}{*}{$\frac{2\left(n+m-2+\frac{1}{\alpha+1}\right)}{n+m-2}$}
    \\
    &$n\ge 4$&  $\mathbb{N}$  &  & \\
   \bottomrule
\end{tabular}
\end{table}

\subsection{The strong unique continuation  property}\label{sec6}

For $\rho>0$ and $y_{0}\in \mathbb{R}^{m}$, we denote
\begin{equation*}
    \begin{split}
      B_{\rho}:=&B_{\rho}((0, y_{0}))\\
      =&\left\{(x,y)\in \mathbb{R}^{n+m} \bigg| \left(|x|^{2(\alpha+1)}+(\alpha+1)^{2}|y-y_{0}|^{2}\right)^{\frac{1}{2(\alpha+1)}}<\rho\right\}.  
    \end{split}
\end{equation*}
Let $S^{2}(B_{\rho})$ be the closure of $C_{0}^{\infty}(B_{\rho})$ under the norm
\begin{eqnarray*}
\|u\|_{S^{2}(B_{\rho})}=\left\{\int_{B_{\rho}}\left(|\nabla_{x}^{2}u|^{2}+\left||x|^{2\alpha}\nabla_{y}^{2}u\right|^{2}
+|\nabla_{x} u|^{2}+\left||x|^{\alpha}\nabla_{y}u\right|^{2}+|u|^{2}\right)\,{\rm d}x\, {\rm d}y\right\}^{\frac{1}{2}}.
\end{eqnarray*}  
Thanks to the Sobolev inequality for the Grushin operator (see \cite[Eq.\,(1.3)]{mor}), 
we have $u\in L^{p_{0}}(B_{\rho})$ if $u\in S^{2}(B_{\rho})$,  where $p_{0}=2Q/(Q-2)$ in which $Q$ is the homogeneous dimension  in \eqref{hdm}. 

The strong unique continuation property is based on the following concept, see e.g. \cite{Ga, GS}.

\begin{definition} We say that function $u$ vanishes of infinite order at the point $(0, y_{0})$ in the $L^{p}$ mean, if for all $N>0$,
\begin{eqnarray*}
\int_{B_{\rho}((0,y_{0}))}|u|^{p}\,{\rm d}x \,{\rm d}y=\mathcal{O}(\rho^{N}), \qquad \mbox{as} \quad \rho\rightarrow 0.
\end{eqnarray*}
\end{definition}

Now, using  the Carleman estimates obtained in Section \ref{sec5} and the aforementioned Sobolev inequality, we can give the strong unique continuation  property for the Schr\"odinger operators $-\Delta_{\alpha}+V$ at points of the degeneracy manifold $\{(0, y)\in \mathbb{R}^{n+m}|y\in \mathbb{R}^{m}\}$. In particular, we determine the allowed potential space  $L_{{\rm loc}}^{r}$. Since it is obtained  by adapting the proof in \cite[Theorem 6.4]{GS} with some small changes,  we do not repeat the details of the proof here.
\begin{theorem}\label{fff1} Suppose that $u\in S^{2}(B_{\rho_{0}}((0, y_{0})))$ for some $\rho_{0}>0$ and $y_{0}\in \mathbb{R}^{m}$. Also, assume that
\begin{eqnarray}\label{ef3}
\left|\Delta_{x}u+|x|^{2\alpha}\Delta_{y}u\right|\le |Vu|    \qquad \mbox{in}\quad B_{\rho_{0}}=B_{\rho_{0}}((0, y_{0}))
\end{eqnarray}
for  potential $V \in L_{{\rm loc}}^{r}$, where $0\le 1/r<1/p-1/q$ with $p,q$ given in Theorem \ref{sxc1} and \ref{sxc2}. Then $u\equiv 0$ in $B_{\rho_{0}}$ if $u$ vanishes of infinite order  at $(0, y_{0})$ in the $L^{2}$ mean.
\end{theorem}

The exponents $p, q$ in the Carleman estimates and  $r$ for the potential space are collected in Table
\ref{tab:example2}.

\section*{Conclusions}
 We have constructed an orthogonal basis of Grushin-harmonics  on $\mathbb{R}^{n+m}$ with $n,m\ge 2$ in Section \ref{31}. However, due to conditions in the polar coordinates  \eqref{pc} and \eqref{jdq}, the cases for $n=1$ and $m\ge 1$ deserve to be discussed separately.  As far as we  see, it is closely related with the so-called generalized Gegenbauer polynomials arising in the study of Dunkl operator on the line \cite{dY}. It is worth mentioning  that the $L^{2}-L^{2}$-weighted estimates for the cases where $n=2$, $m\ge 3$ odd and $\alpha\ge 2$ are not established yet. Finally, besides the potential applications in the study of unique continuation properties on $H$-type groups, our bounds in Section \ref{sec4} will  be useful to establish unique continuation results for the Schr\"odinger operators $-\Delta_{\alpha}+V$ with singular potentials, for instance the Hardy-type potentials, as studied in the Euclidean setting.

\section*{Acknowledgements}
 We thank   Walter Van Assche  and Yuan Xu  for their inspirational discussions on the weighted $L^{2}$ norm estimates of Gegenbauer polynomials. The second author was supported  by NSFC Grant No.12101451 and  China Scholarship Council.

\end{document}